%% file: NoteAK2+.tex
\documentclass[runningheads]{svjour3}

\usepackage{amsmath,amssymb}
\usepackage{cite}
\usepackage[width=12.9cm, left=4.05cm, height=18.59cm] {geometry}
\usepackage{color}
\usepackage{hyperref}
\hypersetup{
    colorlinks=true,
    linkcolor=blue, 
    citecolor=red, 
    urlcolor=blue  } 
\usepackage[colorinlistoftodos]{todonotes}
\usepackage[weather,alpine,misc,geometry]{ifsym}
\smartqed  

\input{macro}

\title{
About Extensions of the Extremal Principle
\thanks{The research was partially supported by the Australian Research Council, project DP160100854.
Hoa T. Bui is supported by an Australian Government Research Training Program (RTP) Stipend and RTP Fee-Offset Scholarship through Federation University Australia.}
}

\author{
Hoa T. Bui
\and
Alexander Y. Kruger
}

\institute{
Hoa T. Bui
\and
Alexander Y. Kruger (\Letter\,)
\at
Centre for Informatics and Applied Optimization,
Faculty of Science and Technology,
Federation University Australia, POB 663, Ballarat, Vic, 3350, Australia\\
\email{h.bui@federation.edu.au, a.kruger@federation.edu.au}
}
\dedication{Dedicated to Professor Michel Th\'era on the occasion of his 70$^{\,th}$ birthday}
\date{Received: date / Accepted: date}

\journalname{Vietnam Journal of Mathematics}
\begin{document}

\maketitle

\begin{abstract}
In this paper, after recalling and discussing the conventional extremality, local extremality, stationarity and approximate stationarity properties of collections of sets and the corresponding (extended) extremal principle, we focus on extensions of these properties and the corresponding dual conditions with the goal to refine the main arguments used in this type of results, clarify the relationships between different extensions and expand the applicability of the generalized separation results.
We introduce and study new more universal concepts of \emph{relative extremality} and \emph{stationarity} and formulate the \emph{relative extended extremal principle}.
Among other things, certain stability of the relative approximate stationarity is proved.
Some links are established between the relative extremality and stationarity properties of collections of sets and (the absence of) certain regularity, lower semicontinuity and Lipschitz-like properties of set-valued mappings.
\end{abstract}




\section{Introduction}\label{intro}
Starting with the pioneering work by Dubovitskii and Milyutin \cite{DubMil65_}, it has become natural, when dealing with optimization and other related problems, to reformulate optimality and other properties under investigation as a kind of extremal behaviour of certain collections of sets.
The concept of a finite \emph{extremal collection of sets} (Definition~\ref{D1}) represents a very general model embracing many optimality notions.
It was first introduced in \cite{KruMor80.2}
and studied in \cite{Kru81.2,KruMor80.2,KruMor80,Kru85.1,Kru03,MorSha96, Iof98,BorJof98,Mor06.1} and many other publications.
The \emph{extremal principle} (Theorem~\ref{EP}) providing a dual characterization of (local) extremality in the form of \emph{generalized separation} dates back to the 1980 paper by Kruger and Mordukhovich \cite{KruMor80.2}.
It can be considered as a generalization of the convex \emph{separation theorem} to nonconvex sets and serves as a powerful tool for deducing optimality conditions in nonconvex optimization and subdifferential and coderivative calculus rules; cf. \cite{Kru85.1,Kru03,Mor06.1}.

Similar to the classical analysis, besides extremality, the concepts of \emph{stationarity} and \emph{regularity} of collections of sets play an important role in this type of analysis and have been intensively investigated in recent years.
It was established in \cite{Kru04,Kru03} that the conclusion of the extremal principle (the generalized separation) actually characterizes a much weaker than local extremality property of \emph{approximate stationarity} (Definition~\ref{D2}).
It became possible, keeping basically the original proof, to formulate (in the Asplund space setting) the \emph{extended extremal principle}: the generalized separation is equivalent to the approximate stationarity (Theorem~\ref{EEP}).
The negation of the approximate stationarity happens to be an important regularity/transversality property of collections of sets known under various names \cite{Kru05,Kru06,Kru09,KruTha13,KruLukTha2} (A table illustrating the evolution of the terminology can be found in \cite[Section~2]{KruLukTha2}.) and closely connected with the fundamental property of \emph{metric regularity} of \SVM s.

Motivated by applications, there have been two independent attempts recently to single out the core part of the conventional proof of the extremal principle and formulate it as a separate statement with the objective to produce a more universal tool, applicable in situations where the conventional (extended) extremal principle fails: \cite[Theorem~3.1]{KruLop12.1} by Kruger and L\'opez and \cite[Lemmas~2.1 and 2.2]{ZheNg06} by Zheng and Ng.
The first result served as a tool when extending the extremal principle to infinite collections of sets, while the lemmas from \cite{ZheNg06} were used when proving fuzzy multiplier rules in set-valued optimization problems.
These lemmas have been further refined and strengthened in Zheng and Ng \cite[Theorems~3.1 and 3.4]{ZheNg11} and Zheng et al. \cite[Theorem~1.1]{ZheYanZou17}.
The mentioned statements in \cite{ZheNg06,ZheNg11,ZheYanZou17}, in particular, waive the traditional for the extremal principle and its extensions in \cite{Kru04,Kru03, KruLop12.1} assumption that the sets have nonempty intersection.
Moreover, it has been observed in \cite{ZheNg11} that the conventional generalized separation condition can be strengthened by adding an additional condition (see condition \eqref{ZN-3}) determining the `directions' of the dual space vectors.
This additional condition is important, for instance, when recapturing the classical separation theorems.

In the current paper we study arbitrary (not necessarily smooth or convex) sets in a normed linear space.
After recalling and discussing the conventional extremality, local extremality, stationarity and approximate stationarity properties of collections of sets and the corresponding (extended) extremal principle, we focus on extensions of these properties and the corresponding dual conditions.
The existing and some new extensions are considered with the goal to refine the main arguments used in this type of results, clarify the relationships between different extensions and expand the applicability of the generalized separation results.

We compare in detail the assumptions and conclusions in \cite[Theorem~3.1]{KruLop12.1} and \cite[Theorem~3.4]{ZheNg11} and show (Proposition~\ref{P3.4}) that an appropriate reformulation of the latter theorem is a consequence of the first one.
We also show (Corollary~\ref{C2.2}) that the main assertions in \cite[Theorem~1.1]{ZheYanZou17} are consequences of the conventional extremal principle.
At the same time, we demonstrate (Corollary~\ref{NLEP}) that \cite[Theorem~3.4]{ZheNg11} is strong enough to recapture the nonlocal extremal principle, although it does not seem to be able to recapture the full local extremal principle.
We briefly discuss (Remark~\ref{R4.2}.3) the role of the additional condition \eqref{ZN-3} and observe that it comes from subdifferentiating a norm at a nonzero point, and analogues of this condition are implicitly present in    the proofs of the conventional extremal principle and its extensions.
We admit the importance of conditions of the \eqref{ZN-3} type in generalized separation statements, but
in the current paper, keeping in line with the conventional formulations and for the sake of simplicity of the presentation, we avoid adding such conditions to the statements.

Unlike the conventional extremal principle and its extensions in \cite{Kru04,Kru03,KruLop12.1} assuming that the sets have a common point, in \cite{ZheNg06,ZheNg11,ZheYanZou17} the intersection of the sets is assumed to be empty, and each set is considered near its own point.
This seems to be an important advancement, which in fact exploits the original ideas behind the conventional extremal principle.
We demonstrate that the case of sets with empty intersection can still be treated within the conventional framework.
This new point of view on the extremal principle is made explicit and further developed in the current paper introducing and studying the new more universal concepts of \emph{relative extremality} and \emph{stationarity}.
We formulate the \emph{relative extended extremal principle} (Theorem~\ref{T4.3}) and a `relative' version of \cite[Theorem~3.1]{KruLop12.1} (Theorem~\ref{T4.2}).
Among other things, certain stability of the relative approximate stationarity is proved (Proposition~\ref{P4.3}).
Some links are established between the relative extremality and stationarity properties of collections of sets and (the absence of) certain regularity, lower semicontinuity and Lipschitz-like properties of set-valued mappings (Proposition~\ref{P4.7} and Remark~\ref{R4.6}).
As a consequence, we demonstrate a connection between the extremality and stationarity properties of collections of sets and the \emph{nonconvex separation property} by Borwein and Jofre \cite{BorJof98} (Proposition~\ref{P4.8}).

The structure of the paper is as follows.
The next Section~\ref{pre} contains some preliminary facts used throughout the paper.
In Section~\ref{S3} we recall and discuss the conventional definitions of extremality, local extremality, stationarity and approximate stationarity of pairs of sets, the conventional extremal principle and its extensions.
The section contains some comparisons, illustrative examples and detailed historical comments.
Section~\ref{S4} is devoted to further extensions of the extremal principle.
It contains a comparison of the assumptions and conclusions in \cite[Theorem~3.1]{KruLop12.1} and \cite[Theorem~3.4]{ZheNg11},
a study of the new concepts of extremality and stationarity relative to given points in each of the sets, the \emph{relative extended extremal principle}, a `relative' version of \cite[Theorem~3.1]{KruLop12.1}, and a discussion of the links between the relative extremality and stationarity properties of collections of sets and (the absence of) certain regularity, lower semicontinuity and Lipschitz-like properties of set-valued mappings.

For simplicity, throughout the paper, we stick to the case of two nonempty sets, the general case of $n$ ($n>1$) sets not being strongly different.
When formulating dual conditions (the extremal principle and its extensions), again for simplicity, only the Asplund space setting is considered.
Recall that a Banach space is \emph{Asplund} if every continuous convex function defined on an open convex set $D$ is Fr\'echet differentiable at each point of some dense subset of $D$ \cite{Phe93}, or equivalently, if the dual of each its separable subspace is separable.
We refer the reader to \cite{Phe93,Mor06.1,BorZhu05} for discussions about and characterizations of Asplund spaces.
All reflexive, in particular, all finite dimensional Banach spaces are Asplund.
By now it is well understood that extensions of the main results to broader classes of (or general) Banach spaces only require substituting in the proofs the Fr\'echet subdifferential sum rule with a sum rule for appropriate 
subdifferentials valid in such spaces.
For instance, in general Banach spaces one can use Clarke subdifferentials or the classical convex subdifferentials if the sets are convex.
One can also define certain abstract subdifferentials formulating the needed properties as axioms; see e.g. \cite{KruLop12.1}.
These are purely straightforward technical tricks which do not involve essentially new ideas.

\section{Preliminaries}\label{pre}


Our basic notation is standard, see e.g. \cite{Mor06.1,RocWet98,DonRoc14}.
Throughout the paper, $X$ is a normed linear space.
Its topological dual is denoted by $X^*$ while $\langle\cdot,\cdot\rangle$ denotes the bilinear form defining the pairing between the two spaces.
The closed unit balls in $X$ and $X^*$ are denoted by $\B$ and $\B^*$, respectively.
$\B_\de(x)$ denotes the open ball with radius $\de>0$ and center $x$.
We use the same symbol $\|\cdot\|$ to denote norms in all normed linear spaces (primal and dual).
If not explicitly stated otherwise, products of normed linear spaces are assumed to be equipped with the maximum norm: $\|(x,y)\|=\max\{\|x\|,\|y\|\}$, $(x,y)\in X\times Y$.
For brevity, we sometimes write $\|x,y\|$ and ${\B}_\de(x,y)$ instead of $\|(x,y)\|$ and ${\B}_\de((x,y))$, respectively.
Given a nonempty subset $A$ of a normed linear space, $\Int A$ and $\bd A$ stand, respectively, for its interior and boundary;
$d(x,A):=\inf_{a\in A}\|x-a\|$ is
the distance from a point $x$ to $A$.
We use the notation $\{A,B\}$ when referring to the pair of sets $A$ and $B$ as a single object.
$\mathbb{N}$ stands for the set of all positive integers.

A set-valued mapping $F:X\rightrightarrows Y$ between two sets $X$ and $Y$ is a mapping, which assigns to every $x\in X$ a subset (possibly empty) $F(x)$ of $Y$.
We use the notations
$\gph F:=\{(x,y)\in X\times Y\mid
y\in F(x)\}$ and $\dom F:=\{x\in X\mid F(x)\ne\es\}$
for the graph and the domain of $F$, respectively, and $F\iv : Y\rightrightarrows X$ for the inverse of $F$.
This inverse (which always exists with possibly empty values at some $y$) is defined by
$F\iv(y) :=\{x\in X|\, y\in F(x)\}$, $y\in Y$.
Obviously, $\dom F\iv=F(X)$.


Dual characterizations of extremality/stationarity (generalized separation) are formulated in this paper in terms of dual tools -- Fr\'echet normal cones.
Recall \cite{Kru03} that, given a subset $A$ of a normed linear space $X$ and a point $a\in A$, the \emph{Fr\'echet normal cone} to $A$ at $a$ is defined as follows:
\begin{gather}\label{NC}
N_{A}(a):= \left\{x^\ast\in X^\ast\mid
\limsup_{x\to a,\,x\in A\setminus\{a\}} \frac {\langle x^\ast,x-a\rangle}
{\|x-a\|} \le 0 \right\}.
\end{gather}
It is a nonempty
closed convex cone, often trivial
($N_{A}(\bx)=\{0\}$).
If $A$ is a convex set, then \eqref{NC} reduces to the normal cone in the sense of convex analysis:
\begin{gather*}\label{CNC}
N_{A}(a):= \left\{x^*\in X^*\mid \langle x^*,x-a \rangle \leq 0 \qdtx{for all} x\in A\right\}.
\end{gather*}
Similarly, given a function $f:X\to\R_\infty:=\R\cup\{+\infty\}$ and a point $a\in A$ with $f(a)<\infty$, the
closed convex set
\begin{gather*}
\partial{f}(a) = \left\{x^\ast\in X^\ast\mid
\liminf\limits_{x\to a,\;x\ne a}
\frac{f(x)-f(a)-\langle{x}^\ast,x-a\rangle}
{\norm{x-a}} \ge 0 \right\}
\end{gather*}
is the \emph{Fr\'echet subdifferential} of $f$ at $a$.
It reduces to the classical Moreau--Rockafellar  subdifferential when $f$ is convex.
The following $\eps$-extension ($\eps\ge0$) of \eqref{NC} is used in the sequel:
the set of \emph{$\eps$-normal elements} to $A$ at $a\in A$:
\begin{gather}\label{NC1}
N_\eps(a\mid A):= \left\{x^*\in X^*\mid \limsup_{x\to a,\,x\in A\setminus\{a\}} \frac {\langle x^*,x-a\rangle}{\|x-a\|} \leq\eps\right\}.
\end{gather}
When $\eps=0$, it reduces to \eqref{NC}.
It is easy to check that $N_\eps(a\mid A)\supset N_A(a)+\eps\B$ for any $\eps\ge0$, and if $A$ is not convex, the inclusion can be strict (see \cite{Kru81.1}).

The following simple lemma used several times throughout the paper provides connections between two common ways of formulating `generalized separation' in terms of normal cones.
It is present implicitly in several existing proofs of dual conditions in the literature.

\begin{lemma}\label{L1}
Let $K_1$ and $K_2$ be nonempty cones in a normed linear space and $\eps\in(0,1)$.
\begin{enumerate}
\item
Suppose vectors $z_1$ and $z_2$ satisfy the conditions:
\begin{gather*}
\norm{z_1}+\norm{z_2}=1,\quad
\norm{z_1+z_2}<\eps,\quad
z_1\in K_1,\quad
z_2\in K_2.
\end{gather*}
Then there exist vectors $\hat z_1$ and $\hat z_2$ satisfying the following conditions:
\begin{gather*}
\norm{\hat z_1}+\norm{\hat z_2}=1,\quad
\hat z_1+\hat z_2=0,\quad
d(\hat z_1,K_1)<\frac{\eps}{2(1-\eps)},\quad
d(\hat z_2,K_2)<\frac{\eps}{2(1-\eps)}.
\end{gather*}
\item
Suppose vectors $z_1$ and $z_2$ satisfy the conditions:
\begin{gather*}
\norm{z_1}+\norm{z_2}=1,\quad
z_1+z_2=0,\quad
d(z_1,K_1)+d(z_2,K_2)<\eps.
\end{gather*}
Then there exist vectors $\hat z_1$ and $\hat z_2$ satisfying the following conditions:
\begin{gather}\label{L1-1}
\norm{\hat z_1}+ \norm{\hat z_2}=1,\quad
\norm{\hat z_1+\hat z_2}<\frac{\eps}{1-\eps},\quad
\hat z_1\in K_1,\quad
\hat z_2\in K_2.
\end{gather}
\end{enumerate}
\end{lemma}
\begin{proof}
(i) Set
\begin{gather*}
z_1':=\frac{z_1-z_2}{2},\quad z_2':=\frac{z_2-z_1}{2}.
\end{gather*}
We have
\begin{gather}\notag
z_1'+z_2'=0,\quad \norm{z_1'-z_1}=\norm{z_2'-z_2} =\frac{1}{2}\norm{z_1+z_2}<\frac{\eps}{2},
\\\label{L1P1}
\norm{z_1'}+\norm{z_2'}\ge\norm{z_1}+\norm{z_2} -\norm{z_1'-z_1}-\norm{z_2'-z_2}>1-\eps.
\end{gather}
Now set
\begin{gather}\label{L1P2}
\hat z_1:=\frac{z_1'}{\norm{z_1'}+\norm{z_2'}},\quad
\hat z_2:=\frac{z_2'}{\norm{z_1'}+\norm{z_2'}}.
\end{gather}
Then $\hat z_1+\hat z_2=0$, $\norm{\hat z_1}+\norm{\hat z_2}=1$ and, for $i=1,2$,
\begin{gather*}
d(\hat z_i,K_i)\le\norm{\hat z_i-\frac{z_i} {\norm{z_1'}+\norm{z_2'}}} =\frac{\norm{z_i'-z_i}}{\norm{z_1'}+\norm{z_2'}} <\frac{\eps}{2(1-\eps)}.
\end{gather*}

(ii) There exist vectors $z_1'\in K_1$ and $z_2'\in K_2$ such that
$\norm{z_1-z_1'}+\norm{z_2-z_2'}<\eps$.
Then condition \eqref{L1P1} is satisfied,
\begin{gather*}
\norm{z_1'+z_2'}\le\norm{z_1+z_2} +\norm{z_1-z_1'}+\norm{z_2-z_2'}<\eps,
\end{gather*}
and the vectors $\hat z_1$ and $\hat z_2$ defined by \eqref{L1P2} satisfy all the conditions in \eqref{L1-1}.
\qed\end{proof}

\section{Extremality, stationarity and extremal principle}\label{S3}

In this section we recall and discuss the conventional definitions of extremality, local extremality, stationarity and approximate stationarity of pairs of sets, the conventional and extended extremal principles.
\subsection{Extremality}

\begin{definition}[Extremality] \label{D1}
Suppose $X$ is a normed linear space, $A,B\subset X$ and $A\cap B\ne\es$.
\begin{enumerate}
\item
The pair $\{A,B\}$ is
\emph{extremal} if
for any $\varepsilon>0$ there exist $u,v\in{X}$
such that
\begin{gather}\label{D1-1}
(A-u)\cap(B-v)=\emptyset
\qdtx{and}
\max\{\|u\|,\|v\|\}<\eps;
\end{gather}
\item
The pair $\{A,B\}$ is
\emph{locally extremal} at $\bx\in A\cap B$ if
there exists a $\rho>0$ such that for any $\varepsilon>0$ there
are $u,v\in X$ such that
\begin{gather}\label{D1-2}
(A-u)\cap (B-v)\cap{\B}_\rho(\bar{x})
=\emptyset
\qdtx{and}
\max\{\|u\|,\|v\|\}<\eps.
\end{gather}
\end{enumerate}
\end{definition}

Condition (i) (condition (ii)) in Definition~\ref{D1} means that an appropriate arbitrarily small shift of the sets makes them nonintersecting (in a neighbourhood of $\bx$).
This is a very general model embracing many optimality notions.
It is easy to see that, if a pair $\{A,B\}$ is extremal, it is locally extremal at any point in $A\cap B$, and the converse is true if $A$ and $B$ are convex.
At the same time, the (nonlocal) extremality in condition (i) can be considered as a special case of the local extremality in condition (ii) with $\rho=\infty$.

The next example illustrates the difference between the extremality and the local extremality.

\begin{example}\label{E3.1}
1. The sets $A:=\{(x_1,x_2)\mid x_2\le0\}$ and $B:=\{(x_1,x_2)\mid x_1^2\le x_2\}$ in $\R^2$ (see Fig.~\ref{F1}) are obviously extremal.

2. If the set $A$ above is modified slightly: $A:=\{(x_1,x_2)\mid x_2\le0\;\;\mbox{or}\;\; x_1\le-1\}$ (see Fig.~\ref{F2}), then $\{A,B\}$ is not extremal any more.
At the same time, it is still locally extremal at $(0,0)\in A\cap B$ (but not at $(-1,1)$!).
\end{example}

\begin{figure}[!ht]
\centering
\begin{minipage}[b]{0.3\textwidth}
\centering
\includegraphics[height=2.5cm]{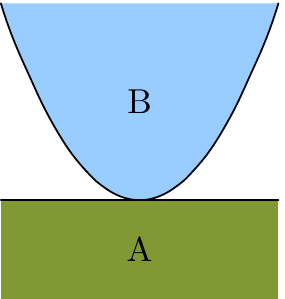}
\caption{Example~\ref{E3.1}.1}\label{F1}
\end{minipage}
\begin{minipage}[b]{0.3\textwidth}
\centering
\includegraphics[height=2.5cm]{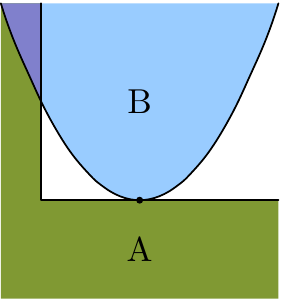}
\caption{Example~\ref{E3.1}.2}\label{F2}
\end{minipage}
\end{figure}

As the next proposition shows, when the sets are closed, the first condition in \eqref{D1-1} can be reformulated equivalently in terms of distances.

\begin{proposition}[Distance characterizations of extremality]
\label{P3.1}
Suppose $X$ is a normed linear space, $A,B\subset X$ are closed and $A\cap B\ne\es$.
The pair $\{A,B\}$ is extremal if and only if for any $\eps>0$ there exist $u,v\in{X}$
such that $\max\{\|u\|,\|v\|\}<\eps$ and the following two equivalent conditions hold:
\begin{enumerate}
\item
$d(a-u,B-v)>0$ for all $a\in A$;
\item
$d(b-v,A-u)>0$ for all $b\in B$.
\end{enumerate}
\end{proposition}

\begin{proof}
It is sufficient to show that each of the conditions (i) or (ii) is equivalent to $(A-u)\cap(B-v)=\emptyset$.
Each of these conditions obviously implies $(A-u)\cap(B-v)=\emptyset$.
Conversely, if $(A-u)\cap(B-v)=\emptyset$, then $a-u\notin B-v$ for any $a\in A$, and consequently, since $B$ is closed, $d(a-u,B-v)>0$, i.e. condition (i) is satisfied.
Similarly, since $A$ is closed, condition $(A-u)\cap(B-v)=\emptyset$ implies $d(b-v,A-u)>0$ for all $b\in B$, hence, condition (ii).
\qed\end{proof}

The closedness assumption in Proposition~\ref{P3.1} cannot be dropped.

\begin{example}
The pair of sets $A:=\R^2\setminus\{(t,0)\mid t>0\}$ and $B:=\{(0,0)\}$ in $\R^2$ is obviously extremal in the sense of Definition~\ref{D1}(i).
At the same time, $d(a-u,B-v)=d(b-v,A-u)=0$ for all $a\in A$, $b\in B$ and $u,v\in\R^2$.
\end{example}

\begin{remark}
Condition $(A-u)\cap(B-v)=\emptyset$ which is crucial for the extremality property in Definition~\ref{D1} is obviously implied by the stronger condition $d(A-u,B-v)>0$, which is also stronger than each of the conditions (i) or (ii) in Proposition~\ref{P3.1}.
As the next example shows, condition
$(A-u)\cap(B-v)=\emptyset$ in the definition of extremality cannot be replaced by condition $d(A-u,B-v)>0$ even when both $A$ and $B$ are closed.
\end{remark}

\begin{example}
Consider two sets in $l^\infty$:
\begin{align*}
A:=&\Big\{x=(x^k)\mid x^K\in [K,K+1]\cup \left[K+1+\frac{1}{K},K+2\right]\;\;\mbox{for some}\;\;K\in\N;
\\
&\hspace{1.7cm} x^k\in [-1,1]\;\;\mbox{for all}\;\; k\neq K\Big\},
\\
B:=&\left\{x=(x^k)\mid x^K=K+1\;\;\mbox{for some}\;\;K\in\N;\; x^k=0\;\;\mbox{for all}\;\; k\neq K\right\}.
\end{align*}
Observe that $A\cap B=B\ne\es.$
We are going to show that $A$ and $B$ are closed, $\{A,B\}$ is extremal, and $d(A-u,B-v)=0$ for all $u,v\in l^\infty$ such that $\max\{\|u\|,\|v\|\}<1/2$.

We first show that $A$ is closed.
Let $(x_n)\subset A$ and $x_n\to x_0\in l^\infty$.
There exist numbers $K,N\in\N$ such that $x_{n}^{K}\in  [K,K+1]\cup [K+1+1/K,K+2]$ for all $n>N.$
Indeed, assume on the contrary that for any $N>0$ there exist $m,n>N$, $m\ne n$ and $K_m,K_n\in\N$, $K_n\neq K_m$ such that $x_{m}^{K_m}\in  [K_m,K_m+1]\cup [K_m+1+1/K_m,K_m+2]$ and $x_{n}^{K_n}\in  [K_n,K_n+1]\cup [K_n+1+1/K_n,K_n+2]$. Then
\begin{align*}
\|x_n-x_m\|&\ge \max\left\{|x_{n}^{K_n}-x_{m}^{K_n}|, |x_{n}^{K_m}-x_{m}^{K_m}|\right\}
\\&
\ge \max\left\{|x_{n}^{K_n}|-|x_{m}^{K_n}|, |x_{m}^{K_m}|-|x_{n}^{K_m}|\right\}\ge \max\{K_n,K_m\}-1\ge1,
\end{align*}
which contradicts the assumption that $(x_n)$ is convergent.
Hence, $x_0=(x_{0}^{1},x_{0}^{2},\ldots)$ with $x_{0}^{K}\in [K,K+1]\cup [K+1+1/K,K+2]$ and $x_{0}^{k} \in [-1,1]$ for all $k\neq K$. Thus, $A$ is closed.

A similar argument can be used to show that $B$ is closed.
Observe that for all $x_1,x_2\in B$ with $x_1\neq x_2$ one has $\|x_1-x_2\| = \max\{K_1,K_2\}+1$ for some $K_1,K_2\in\N$, $K_1\ne K_2$.
Hence, $\|x_1-x_2\|>1$.
It follows that any convergent sequence $(x_n)\subset B$ must be stationary when $n$ is sufficiently large.
This immediately yields the closedness of $B$.

Now we show that $\{A,B\}$ is extremal.
Given an $\varepsilon\in(0,1),$ find and an $n\in\N$ such that $1/n<\varepsilon\le1/(n-1)$ and define a $u\in l^\infty$ as follows: $u^i = 1/n$ if $i<n$, and $u^i=1/(i+1)$ if $i\ge n$.
We have $\|u\|\le\frac{1}{n}<\varepsilon$.
Let $b=(b^k)\in B$, i.e. there exists a $K\in\N$ such that $b^K= K+1$ and $b^k=0$ for all $k\ne K$.
Then $0<(b+u)^k\le1/n<1$ for all $k\ne K$.
If $K<n$, then $(b+u)^K=K+1+1/n$.
If $K\ge n$, then $(b+u)^K=K+1+1/(K+1)$.
In both cases, $K+1<(b+u)^K<K+1+1/K$.
Hence, $b+u\notin A$, and consequently, $(A-u)\cap B=\es$.

Let $u=(u^k),v=(v^k)\in l^\infty$ be such that $\max\{\|u\|,\|v\|\}<1/2$ and $(A-u)\cap(B-v)=\emptyset$.
We are going to show that $d(A-u,B-v)=0.$
Obviously $\|u-v\|<1$.
Moreover, $|u^k-v^k|<1/k$ for all $k\in\N$.
Indeed, suppose on the contrary that $1/K\le|u^K-v^K|<1$ for some $K\in\N$ and choose a $b=(b^k)\in B$ such that $b^K=K+1$ and $b^k=0$ for $k\ne K$.
Then for any $k\ne K$, we have $\|(b+u-v)^k\|=\|(u-v)^k\|<1$, and $(b+u-v)^K=K+1+u^K-v^K$, and consequently, either $K<(b+u-v)^K<K+1$ or $K+1+1/K\le(b+u-v)^K<K+2$.
In any case, $b+u-v\in A$, and $b-v\in A-u$, which is a contradiction.
Thus, $|u^k-v^k|<1/k$ for all $k\in\N$.
For any $\eps>0$, we can find a $b\in B$ and a $K\in\N$ such that $b^K\neq 0$, $|u^K-u^K|<\eps$.
Set $a^K:=b^K$ and $a^k:=u^k-v^k$ for all $k\ne K$.
Then $a=(a^k)\in A$ and $\|(a-u)-(b-v)\|=|a^K-b^K| <\eps$.
Hence, $d(A-u,B-v)=0.$
\end{example}

\subsection{Extremal principle}

The next well-known theorem (see Remark~\ref{R2.1} below) gives approximate dual necessary conditions of local extremality in terms of Fr\'echet normals.
It can be considered as a generalization of the classical convex \emph{separation theorem} to pairs of nonconvex sets.

\begin{theorem}[Extremal principle]\label{EP}
Suppose $X$ is an Asplund space, $A,B\subset X$ are closed and $\bx\in A\cap B$.
If the pair $\{A,B\}$ is locally extremal at $\bx$, then the following two equivalent conditions hold:
\begin{enumerate}
\item
for any $\eps>0$ there exist points
$a\in A\cap\B_\eps(\bar{x})$, $b\in B\cap\B_\eps(\bar{x})$ and $a^*\in X^*$ such that
\begin{gather}\label{EP-1}
\norm{a^*}=1,
\quad
d(a^*,N_A(a))<\varepsilon
\qdtx{and}
d(-a^*,N_B(b))<\varepsilon;
\end{gather}
\item
for any $\eps>0$ there exist points
$a\in A\cap\B_\eps(\bar{x})$, $b\in B\cap\B_\eps(\bar{x})$, $a^*\in{N}_{A}(a)$ and $b^*\in{N}_B(b)$ such that
\begin{gather}\label{EP-2}
\norm{a^*}+\norm{b^*}=1 \quad\mbox{and}\quad \norm{a^*+b^*}<\eps.
\end{gather}
\end{enumerate}
\end{theorem}

\begin{remark}\label{R3.1}
The inequalities in \eqref{EP-1} and \eqref{EP-2} are only meaningful when $\eps\le1$, because otherwise they are direct consequences of the corresponding equalities.
Indeed, when $\eps>1$, condition (i) in Theorem~\ref{EP} is satisfied automatically while condition (ii) guarantees only the existence of nontrivial normals in the $\eps$-neighbourhood of $\bx$ to at least one of the sets $A$ and $B$, which
is trivial as long as $\bx$ is a boundary point of one of the sets (which is the case when $\{A,B\}$ is locally extremal at $\bx$).
If conditions \eqref{EP-1} and \eqref{EP-2} hold with $\eps=1$, they also hold with some $\eps<1$.
Thanks to these observations, when applying Theorem~\ref{EP} or its extensions, one can always assume that $\eps<1$.
\end{remark}

Both conclusions in the above theorem are pretty common dual space properties used in many contemporary formulations of the extremal principle and its extensions.
Properties (i) and (ii) can be found e.g. in, respectively, \cite[Definition~2.5]{Mor06.1} (the \emph{approximate extremal principle}) and \cite[Definition~2.3]{Kru03} (the \emph{generalized Euler equation}); cf. \cite[property~(SP)$_S$]{Kru09}.
Condition (i) guarantees the existence of a pair of vectors $a^*$ and $b^*$ in the dual space, which are `almost normal' (up to $\eps$) to the corresponding sets at certain points with $a^*+b^*=0$ and $\|a^*\|=\|b^*\|=1$, while condition (ii) guarantees the existence of a pair of vectors $a^*$ and $b^*$ which are exactly normal (in the Fr\'echet sense) to the corresponding sets at certain points with their sum $a^*+b^*$ being small (up to $\eps$) and $\|a^*\|+\|b^*\|=1$.

The equivalence of the two properties is a consequence of Lemma~\ref{L1}.

\begin{proof} \emph{of the equivalence of conditions {\rm(i)} and {\rm(ii)} in Theorem~\ref{EP}}

(i) \folgt (ii).
Take an arbitrary $\eps>0$ and set $\xi:=\frac{\eps}{1+\eps}$.
It follows from (i) that there exist points
$a\in A\cap\B_\xi(\bar{x})$, $b\in B\cap\B_\xi(\bar{x})$  and $a^*\in X^*$ such that conditions \eqref{EP-1} hold true with $\xi$ in place of $\eps$.
Then
\begin{gather*}
\norm{\frac{a^*}{2}}+\norm{-\frac{a^*}{2}}=1\qdtx{and}
d\left(\frac{a^*}{2},N_A(a)\right) +d\left(-\frac{a^*}{2},N_B(b)\right) <\xi.
\end{gather*}
Using Lemma~\ref{L1}(ii), we can find $\hat a^*\in N_A(a)$ and $\hat b^*\in N_B(b)$ such that
$$\|\hat a^*\|+\|\hat b^*\|=1\qdtx{and}
\|\hat a^*+\hat b^*\|<\frac{\xi}{1-\xi} =\eps.$$

(ii) \folgt (i).
Take an arbitrary $\eps>0$ and set $\xi:=\frac{\eps}{1+\eps}$.
It follows from (ii) that there exist points
$a\in A\cap\B_\xi(\bar{x})$, $b\in B\cap\B_\xi(\bar{x})$, $a^*\in{N}_{A}(a)$ and $b^*\in{N}_B(b)$ such that conditions \eqref{EP-2} hold true with $\xi$ in place of $\eps$.
Using Lemma~\ref{L1}(i), we can find $\hat a^*,\hat b^*\in X^*$ such that
\begin{gather*}
\|\hat a^*\|+\|\hat b^*\|=1,\quad
\hat a^*+\hat b^*=0,
\\
d(\hat a^*,N_A(a))<\frac{\xi}{2(1-\xi)}=\frac{\eps}{2}, \quad
d(\hat b^*,N_B(b))<\frac{\xi}{2(1-\xi)}=\frac{\eps}{2}.
\end{gather*}
Then $\|2\hat a^*\|=1$, $d(2\hat a^*,N_A(a))<\eps$ and $d(-2\hat a^*,N_B(b))<\eps$.
\qed\end{proof}

Conditions (i) and (ii) in Theorem~\ref{EP} obviously hold if the pair $\{A,B\}$ is (not necessarily locally) extremal and $\bx\in A\cap B$.

\begin{remark}[Extremal principle: historical comments]\label{R2.1}
The extremality properties in parts (i) and (ii) of Definition~\ref{D1} were originally introduced in \cite{KruMor80.2} (see Definition~4.1 and Remarks~4.1 and 4.8), where their connections with the separation of sets were also discussed and the first version of the extremal principle was established first in finite dimensions in terms of limiting normal cones \cite[Theorem~4.1]{KruMor80.2} and then extended, with the help of the \EVP, to \emph{Fr\'echet smooth} spaces, i.e. Banach spaces admitting an equivalent norm Fr\'echet differentiable away from zero \cite[Theorem~6.1]{KruMor80.2}, in terms of sets of $\eps$-normal elements.
The latter result was formulated in the form similar to (but slightly weaker than) the condition (ii) in Theorem~\ref{EP}:
\begin{enumerate}
\item [(ii)$^\prime$]
for any $\eps>0$ there exist points
$a\in A\cap\B_\eps(\bar{x})$, $b\in B\cap\B_\eps(\bar{x})$, $a^*\in{N}_\eps(a\mid A)$ and $b^*\in{N}_\eps(b\mid B)$ such that conditions \eqref{EP-2} hold true.
\end{enumerate}
In the above condition, ${N}_\eps(a\mid A)$ stands for the set of $\eps$-normal elements \eqref{NC1} to $A$ at $a$.

A slightly weaker version of \cite[Theorem~6.1]{KruMor80.2} (under the stronger assumption Definition~\ref{D1}(i) instead of (ii)) was presented in \cite{KruMor80} accompanied by a short sketch of the proof.

While keeping the original pattern of the proof, the result of \cite[Theorem~6.1]{KruMor80.2} was strengthened in \cite[Theorem~2]{Kru85.1} along two directions: 1) the assumption of the existence of an equivalent Fr\'echet differentiable norm was relaxed to that of the existence of a \nbh\ $U$ of zero and a continuous function $\psi:U\to\R_+$ such that $\psi(x)=0$ if and only if $x=0$, $\psi$ is Fr\'echet differentiable on $U\setminus\{0\}$ with $\|\psi(x)\|\ge1$ for all $x\in U\setminus\{0\}$; and 2) conclusion (ii)$^\prime$ replaced by a stronger one (the \emph{$\eps$-extremal principle} \cite[Definition~2.5]{Mor06.1}):
\begin{enumerate}
\item [(i)$^{\prime}$]
for any $\eps>0$ there exist points
$a\in A\cap\B_\eps(\bar{x})$, $b\in B\cap\B_\eps(\bar{x})$ and $a^*\in{X}$ such that
\begin{gather*}
\norm{a^*}=1,
\quad
a^*\in{N}_\eps(a\mid A)
\quad\mbox{and}\quad
-a^*\in{N}_\eps(b\mid B).
\end{gather*}
\end{enumerate}
Note that this condition is still in general weaker than condition (i) in Theorem~\ref{EP}.

The next important step was made by Mordukhovich and Shao in \cite[Theorem~3.2]{MorSha96} where, using the subdifferential characterizations of Asplund
spaces (the sum rule for Fr\'echet subdifferentials) established by Fabian \cite{Fab86,Fab89}, the extremal principle with minimal adjustments in the original proof was extended to general Asplund spaces.
In particular, it was shown that in Asplund
spaces condition (i)$^{\prime}$ above is equivalent to condition (i) in Theorem~\ref{EP}.
Moreover, it was also shown in \cite[Theorem~3.2]{MorSha96} that Theorem~\ref{EP} in its current form cannot be extended beyond Asplund
spaces.

The last observation raised the question about possible extension of the extremal principle to non-Asplund spaces and the right tools needed for that since the Fr\'echet ($\eps$-)normals cannot do the job.
It did not take long for the experts in this area to pinpoint those properties of normals which are actually used in the conventional proof of the extremal principle.
This led to several successful attempts to formulate these properties as sets of axioms and define several (very similar) abstract normal cones (and related subdifferentials) which could replace the Fr\'echet normal cones without changing much in the conventional proof of the extremal principle; see e.g. \cite{Iof98,BorJof98} and \cite[Subsection~2.5.3]{Mor06.1}.
This allowed extending the extremal principle to \emph{trustworthy} spaces (with respect to a given subdifferential/normal cone) \cite{Iof98} with Asplund spaces being trustworthy with respect to the Fr\'echet subdifferential and general Banach spaces being trustworthy with respect to, e.g., Clarke subdifferential.

Extending dual space results formulated in Asplund spaces in terms of Fr\'echet subdifferentials and normals, including the extremal principle, to general Banach spaces in terms of Clarke or other subdifferentials and normals, for which Banach spaces are trustworthy, has become a straightforward routine procedure.
In this paper for simplicity we restrict the presentation to Asplund spaces and Fr\'echet normals only.

We refer the readers to \cite[Section~2.6]{Mor06.1} for more historical comments.
\end{remark}

\begin{remark}[Extremal principle: two sets vs $n$ sets]\label{R2.2}
The original formulations of the definitions of extremality and extremal principle in \cite{KruMor80.2} and most of their subsequent reformulations and generalizations \cite{KruMor80,Kru81.2,Kru85.1,Mor06.1,MorSha96, BorJof98,BorZhu05} have been for the more general than in Definition~\ref{D1} and Theorem~\ref{EP} setting of $n\ge2$ sets.
This seemingly more general setting is in fact not much different in terms of ideas, proofs and applications from the case of two sets considered for simplicity in the current paper.
Moreover, it is well known (see e.g. \cite[proof of Theorem~6.1]{KruMor80.2}, \cite[p.~31]{Kru81.2}, \cite[proof of Theorem~2]{Kru85.1}, \cite[p.~111 and 112]{Kru05}, \cite[proof of Theorem~2.10]{Mor06.1}) that the case of $n$ sets $A_1,A_2,\ldots,A_n\subset X$ can be easily reduced to that of two sets: either $A:=A_1\times A_2\times\ldots\times A_n$ and $B:=\{(x,\ldots,x)\mid x\in X\}$ in $X^n$ or $A:=A_1\times A_2\times\ldots\times A_{n-1}$ and $B:=\{(x,\ldots,x)\mid x\in A_n\}$ in $X^{n-1}$.
This trick, sometimes referred to as \emph{Pierra's product space reformulation} \cite{Pie84}, is not easily applicable to the case of an infinite collection of sets treated in \cite{KruLop12.1}.

When dealing with arbitrary finite collections of sets, the seemingly weaker property of local extremality of a collection of $n$ sets can be considered as a particular case of the nonlocal extremality of a collection of $n+1$ sets.
\end{remark}

\begin{remark}[Nonlocal extremality]\label{R2.4}
The (nonlocal) extremality property, as defined in part (i) of Definition~\ref{D1}, does not use the assumption $A\cap B\ne\es$, present in the preamble of Definition~\ref{D1} as well as in the original definition of this property in \cite{KruMor80.2}.
The conventional proof of the extremal principle can proceed without this assumption (even getting a little shorter) and leading to a result (nonlocal extremal principle) which differs from the conclusions of Theorem~\ref{EP} below by the conditions
$a\in A\cap\B_\eps(\bar{x})$, $b\in B\cap\B_\eps(\bar{x})$ being replaced simply with $a\in A$, $b\in B$ and
\begin{equation}\label{R2.4-1}
\norm{a-b}<d(A,B)+\eps.
\end{equation}
Note that if $A\cap B=\es$, then $\{A,B\}$ is automatically extremal in the relaxed sense discussed in this remark.
\end{remark}

\begin{remark}[Normalization conditions]\label{R2.3}
The two conditions \eqref{EP-2} in Theorem~\ref{EP} can be replaced by the following single one:
\begin{gather}\label{EP-3}
\norm{a^*+b^*}<\eps\left(\norm{a^*}+\norm{b^*}\right).
\end{gather}
Observe that under condition \eqref{EP-3} vectors $a^*$ and $b^*$ cannot equal zero simultaneously.
The sum of the norms $\norm{a^*}+\norm{b^*}$ in \eqref{EP-2} and \eqref{EP-3} can be replaced by the maximum: $\max\left\{\norm{a^*},\norm{b^*}\right\}$, or more generally, by $|||\left(\norm{a^*},\norm{b^*}\right)|||_*$, where $|||\cdot|||_*$ is an arbitrary norm on $\R^2$.
A similar observation can be made regarding the expression $\max\{\|u\|,\|v\|\}$ in Definition~\ref{D1}, where the maximum can be replaced by the sum: $\|u\|+\|v\|$, or more generally, by $|||\left(\norm{u},\norm{v}\right)|||$, where $|||\cdot|||$ is an arbitrary norm on $\R^2$.
In fact, it would be natural to choose the norms $|||\cdot|||$ and $|||\cdot|||_*$ to be dual to each other.
\end{remark}

Theorem~\ref{EP} formulated for a pair of sets yields the following result for a single set, generalizing (in Asplund spaces) the Bishop-Phelps theorem \cite[Theorem~3.18]{Phe93} (cf. \cite[Corollary~1 from Theorem~2.1]{Kru81.2}, \cite[Corollary~3.4]{MorSha96}, \cite[Corollary~2.12.1]{Kru03}, \cite[Proposition~2.6]{Mor06.1}).

\begin{corollary}[Density of `support' points]\label{EPC}
Suppose $X$ is an Asplund space, $A\subset X$ is closed and $\bx\in\bd A$.
Then,
for any $\eps>0$, there exists a point
$a\in A\cap\B_\eps(\bar{x})$ such that
$N_A(a)\ne\{0\}$.
\end{corollary}

\begin{proof}
The assertion follows by applying Theorem~\ref{EP} to $A$ and $B:=\{\bx\}$.
\qed\end{proof}

As observed in \cite[p.~177]{Mor06.1}, one can also go in the opposite direction: deducing a kind of extremal principle for a pair of sets from the density result for a single set in Corollary~\ref{EPC}.

\begin{proposition}\label{P2.1}
Suppose $X$ is an Asplund space, $A,B\subset X$ and $A-B$ is closed (for instance, both sets are closed and one of them is compact).
If the pair $\{A,B\}$ is extremal, then for any $\eps>0$ there exist points $a\in A$ and $b\in B$ satisfying $\norm{a-b}<\eps$, and an $a^*\in X^*$ such that
\begin{gather}\label{EP-1'}
\norm{a^*}=1,
\quad
a^*\in N_A(a)
\qdtx{and}
-a^*\in N_B(b).
\end{gather}
\end{proposition}
\begin{proof}
Since $\{A,B\}$ is extremal, we have $0\in A-B$ and $0\notin\Int(A-B)$, hence, $0\in\bd(A-B)$.
Given an $\eps>0$, by Corollary~\ref{EPC} there are $a\in A$ and $b\in B$ with $\norm{a-b}<\eps$, and an $a^*\in N_{A-B}(a-b)$ with $\norm{a^*}=1$.
It remains to notice that the inclusion $a^*\in N_{A-B}(a-b)$ implies $a^*\in N_A(a)$ and $-a^*\in N_B(b)$ (see e.g. \cite[Proposition~1.27]{Kru03}).
\qed\end{proof}
\begin{remark}
1. The conditions in \eqref{EP-1'} guaranteed by Proposition~\ref{P2.1} are stronger than the corresponding conditions in \eqref{EP-1} in Theorem~\ref{EP}.
The latter conditions only guarantee that $a^*$ and $-a^*$ are close to $N_A(a)$ and $N_B(b)$, respectively.
On the other hand, unlike Theorem~\ref{EP}, Proposition~\ref{P2.1} cannot relate the points $a\in A$ and $b\in B$ to a particular point in $A\cap B$.

2. The statement of Proposition~\ref{P2.1} can be easily extended to the relaxed version of extremality without the assumption $A\cap B\ne\es$ (see Remark~\ref{R2.4}).
One only needs to replace the inequality $\norm{a-b}<\eps$ in the conclusion by condition \eqref{R2.4-1}.
In the case $A\cap B=\es$, as the next corollary shows, one can make another step and waive the assumption of the extremality of $\{A,B\}$.
\end{remark}

\begin{corollary}\label{C2.2}
Suppose $X$ is an Asplund space, $A,B\subset X$ and $A-B$ is closed (for instance, both sets are closed and one of them is compact).
If $A\cap B=\es$, then for any $\eps>0$ there exist points $a\in A$ and $b\in B$ satisfying condition \eqref{R2.4-1} and an $a^*\in X^*$ satisfying conditions \eqref{EP-1'}.
\end{corollary}
\begin{proof}
Given an $\eps>0$, set $\eps'=\eps/2$.
There exist points $a'\in A$ and $b'\in B$ satisfying $\|a'-b'\|<d(A,B)+\eps'$.
Without loss of generality $a'-b'\in\bd(A-B)$.
Indeed, since $A-B$ is closed and $0\notin A-B$, the one-dimensional set $\{t\in[0,1]\mid t(a'-b')\in A-B\}$ is compact and its infimum is attained at some $\bar t>0$, which means that there exist $a''\in A$ and $b''\in B$ such that $a''-b''=\bar t(a'-b')$ and $a''-b''\in\bd(A-B)$.
Obviously $\|a''-b''\|\le\|a'-b'\|<d(A,B)+\eps'$.

By Corollary~\ref{EPC}, there are $a\in A$ and $b\in B$ with $\|(a-b)-(a'-b')\|<\eps'$, and an $a^*\in N_{A-B}(a'-b')$ with $\norm{a^*}=1$.
It remains to notice that $\|a-b\|<\|a'-b'\|+\eps'<d(A,B)+\eps$, and the inclusion $a^*\in N_{A-B}(a-b)$ implies $a^*\in N_A(a)$ and $-a^*\in N_B(b)$ (see e.g. \cite[Proposition~1.27]{Kru03}).
\qed\end{proof}

Corollary~\ref{C2.2} recaptures the main assertions in \cite[Theorem~1.1]{ZheYanZou17}.

\subsection{Stationarity and extended extremal principle}
The extremal principle in Theorem~\ref{EP} gives necessary conditions of (local) extremality which are in general not sufficient.
Just like in the classical analysis and optimization theory, it actually characterizes a weaker than extremality property which can be interpreted as a kind of \emph{stationarity}.
The properties in the next proposition came to life as a result of a search for the weakest assumptions on the sets $A$ and $B$ which still ensure the conclusions of the extremal principle.


\begin{definition}[Stationarity] \label{D2}
Suppose $X$ is a normed linear space, $A,B\subset X$ and $\bx\in A\cap B$.
\begin{enumerate}
\item
The pair $\{A,B\}$ is
\emph{stationary} at $\bx$ if
for any $\varepsilon>0$ there exist a $\rho\in(0,\eps)$ and
$u,v\in{X}$ such that
\begin{gather}\label{D2-1}
(A-u)\cap (B-v)\cap{\B}_\rho(\bar{x})
=\emptyset
\qdtx{and}
\max\{\|u\|,\|v\|\}<\eps\rho;
\end{gather}
\item
The pair $\{A,B\}$ is
\emph{approximately stationary} at $\bx$ if for any $\varepsilon>0$ there exist $\rho\in(0,\eps)$,
$a\in A \cap{\B}_\eps(\bar{x})$, $b\in B \cap{\B}_\eps(\bar{x})$ and $u,v\in{X}$ such that \begin{gather}\label{D2-2}
(A-a-u)\cap(B-b-v)\cap(\rho\B)=\emptyset
\qdtx{and}
\max\{\|u\|,\|v\|\}<\eps\rho.
\end{gather}
\end{enumerate}
\end{definition}

Unlike \eqref{D1-2}, in conditions \eqref{D2-1} and \eqref{D2-2} the size of the ``shifts'' of the sets is related to that of the neighbourhood in which the sets become nonintersecting, namely $\max\{\|u\|,\|v\|\}/\rho<\eps$.
Compared to \eqref{D2-1}, in conditions \eqref{D2-2}, instead of the common point $\bx$, the sets $A$ and $B$ are considered near their own points $a$ and $b$, respectively.

The implications in the next proposition are immediate consequences of Definitions~\ref{D1} and \ref{D2}, while the equivalences were proved in \cite[Proposition~14]{Kru05}.

\begin{proposition}[Extremality vs stationarity]\label{P2.2}
Suppose $X$ is a normed linear space, $A,B\subset X$ and $\bx\in A\cap B$.
Consider the following properties:
\begin{enumerate}
\item
$\{A,B\}$ is extremal;
\item
$\{A,B\}$ is locally extremal at $\bx$;
\item
$\{A,B\}$ is stationary at $\bx$;
\item
$\{A,B\}$ is approximately stationary at $\bx$.
\end{enumerate}
Then
\rm(i) \folgt (ii) \folgt (iii) \folgt (iv).
If, additionally, $A$ and $B$ are convex, then
\rm(i) \iff (ii) \iff (iii) \iff (iv).
\end{proposition}

All the implications in Proposition~\ref{P2.2} can be strict;
see some examples in \cite{Kru05,Kru09}.

Replacing in Theorem~\ref{EP} local extremality with approximate stationarity produces a stronger statement -- the \emph{extended extremal principle}, with the two equivalent conditions in the conclusion of Theorem~\ref{EP} becoming not only necessary but also sufficient, thus producing full duality.
The proof of the necessity in the next theorem is a refined version of the proof of Theorem~\ref{EP}, while the proof of the sufficiency is a straightforward consequence of the definitions and does not use the assumption of the Asplund property of the space; cf. \cite[Theorem~4.1]{Kru02}.

\begin{theorem}[Extended extremal principle]\label{EEP}
Suppose $X$ is an Asplund space, $A,B\subset X$ are closed and $\bx\in A\cap B$.
The pair $\{A,B\}$ is approximately stationary at $\bx$ if and only if the two equivalent conditions in Theorem~\ref{EP} hold true.
\end{theorem}

\begin{remark}[Extended extremal principle: historical comments]
The approximate stationarity property in part (ii) of Definition~\ref{D2} was originally introduced in a slightly different form in \cite[formula~(4)]{Kru98}, where the property was referred to as \emph{extremality near $\bx$}.
A version of the extended extremal principle was formulated in \cite[Theorem~2]{Kru98} in the setting of a Fr\'echet smooth Banach space in the form of condition (i)$^\prime$ in Remark~\ref{R2.1}.
The property in part (i) of Definition~\ref{D2} was also implicitly present in \cite{Kru98} (see formula~(5)).
The result was extended to Asplund spaces in \cite[Theorem~2]{Kru00}, where it was formulated in the form of condition (ii)$^\prime$ in Remark~\ref{R2.1}.
The full proof of the extended extremal principle in the form of condition (ii) in Theorem~\ref{EP} appeared in \cite[Theorem~4.1]{Kru02}, where the property in part (ii) of Definition~\ref{D2} was referred to as \emph{extended extremality (e-extremality) near $\bx$}.
It was shown in \cite[Theorem~3.7]{Kru03} that in its current form the result cannot be extended beyond Asplund spaces.

In \cite{Kru04,Kru05,Kru06} the properties in Definition~\ref{D2} are referred to as \emph{stationarity} and \emph{weak stationarity}, respectively.
\cite{Kru06} gives a slightly improved version of the definition of the last property, compared to that in \cite{Kru04,Kru05}.
The name \emph{approximate stationarity} for the property in part (ii) of Definition~\ref{D2} appeared in \cite{Kru09}.
Extensions of Theorem~\ref{EEP} to non-Asplund spaces are discussed in \cite{KruLop12.1}.
\end{remark}

\begin{remark}[Extended extremal principle: two sets vs $n$ sets]
Similarly to the case of extremality and extremal principle (see Remark~\ref{R2.2}), the stationarity properties and extended extremal principle are usually formulated for the setting of $n\ge2$ sets with the case of $n$ sets easily reduced the same way to that of two sets (see e.g. \cite[item~3]{Kru98}, \cite[Definition~4]{Kru00}, \cite[Definition~4.2]{Kru02}, \cite[Definition~3.5]{Kru03}, \cite[Proposition~7]{Kru04}), \cite[Proposition~20 and Remark~6]{Kru05}).
Remark~\ref{R2.3} applies entirely to Definition~\ref{D2} and Theorem~\ref{EEP}.
\end{remark}

\begin{remark}[Approximate stationarity vs transversality]
Theorem~\ref{EEP} can be reformulated as equivalence of the negations of the primal and dual properties involved in its statement: the absence of the approximate stationarity is equivalent to the absence of the generalized separation.
These are important \emph{regularity/transversality} properties of pairs of sets involved in constraint qualifications, qualification conditions in subdifferential calculus and convergence analysis of computational algorithms \cite{Kru05,Kru06,Kru09,KruTha13,KruTha15,Iof}.
They are known under various names.
A table illustrating the evolution of the terminology can be found in \cite[Section~2]{KruLukTha2}.
The transversality property of finite collections of sets (the negation of the approximate stationarity) is in a sense equivalent to the famous \emph{metric regularity} property of \SVM s; cf. \cite{Kru05,Kru06,Kru09}.
See also the discussion in Subsection~\ref{SS4.2} below.
\end{remark}

\section{More extensions}\label{S4}

\subsection
{Two recent extensions}
\label{FE}

As it was pointed out in Remark~\ref{R2.1}, the key tool used in the proof of Theorem~\ref{EP} (and also Theorem~\ref{EEP}) is the Ekeland variational principle.
Theorem~\ref{EEP} is in a sense the ultimate version of Theorem~\ref{EP} establishing the same conclusion (generalized separation) under the weakest possible assumptions on the pair of sets (approximate stationarity), thus, providing the complete duality (in the Asplund space setting) between the corresponding primal space and dual space properties.

The next natural step in the extremal principle refinement process is to single out the core part of the conventional proof of the extremal principle around the application of the Ekeland variational principle, identify the minimal assumptions on the sets and the immediate conclusions and formulate it as a separate statement.
Such a result (results) would expose the core arguments behind the extremal principle and could serve as a key building block when constructing other generalized separation statements, applicable in situations where the conventional (extended) extremal principle fails.

We are aware of two recent attempts of this kind: \cite[Theorem~3.1]{KruLop12.1} which served as a tool when extending Theorems~\ref{EP} and \ref{EEP} to infinite collections of sets, and \cite[Lemmas~2.1 and 2.2]{ZheNg06} used when proving fuzzy multiplier rules in set-valued optimization problems.
The last couple of lemmas have been further refined and strengthened in \cite[Theorems~3.1 and 3.4]{ZheNg11} and \cite[Theorem~1.1]{ZheYanZou17}.


The next two theorems are reformulations for the setting adopted in the current paper of \cite[Theorem~3.1]{KruLop12.1} and \cite[Theorem~3.4]{ZheNg11}, respectively.

\begin{theorem}[Kruger and L\'opez, 2012]\label{KL}
Suppose $X$ is an Asplund space, $A,B\subset X$ are closed and $\bx\in A\cap B$.
\begin{enumerate}
\item
If points $a\in A$, $b\in B$ and $u,v\in X$ satisfy conditions \eqref{D2-2}
with some $\eps>0$ and $\rho>0$,
then, for any $\de>\max\{\|a-\bx\|,\|b-\bx\|\}+\rho(\eps+1)$, there exist points
$a'\in A\cap\B_\de(\bar{x})$, $b'\in B\cap\B_\de(\bar{x})$ and $a^*\in{N}_{A}(a')$, $b^*\in{N}_B(b')$ satisfying conditions \eqref{EP-2}.
\item
If $a\in A$, $b\in B$ and
$a^*\in{N}_{A}(a)$, $b^*\in{N}_B(b)$ satisfy
conditions \eqref{EP-2}
for some $\eps>0$,
then, for any $\de>0$, there exists a $\rho\in(0,\de)$ and points
$u,v\in{X}$ satisfying conditions \eqref{D2-2}.
\end{enumerate}
\end{theorem}

\begin{theorem}[Zheng and Ng, 2011]\label{ZN}
Suppose $X$ is an Asplund space, $A,B\subset X$ are closed, $A\cap B=\emptyset$.
If points $a\in A$ and $b\in B$ satisfy condition \eqref{R2.4-1} with some $\eps>0$,
then, for any $\la>0$ and $\tau\in(0,1)$, there exist points
$a'\in A\cap\B_{\la}(a)$, $b'\in B\cap\B_{\la}(b)$ and $a^*\in{X}^*$ such that
\begin{gather}\label{ZN-2}
\|a^*\|=1,
\quad
d(a^*,N_A(a'))+d(-a^*,N_B(b'))<\varepsilon/\lambda,
\\\label{ZN-3}
\tau\|a'-b'\|\le\langle a^*,a'-b'\rangle.
\end{gather}
\end{theorem}

First observe that the extremal principle in Theorem~\ref{EEP} is a direct corollary of Theorem~\ref{KL}.

\begin{proof}[Theorem~{\rm \ref{EEP}} from Theorem~{\rm \ref{KL}}]
Let the pair $\{A,B\}$ be approximately stationary at $\bx$.
We are going to show that condition (ii) in Theorem~\ref{EP} holds true.
Given an $\eps>0$, find an $\eps'>0$ such that $\eps'(\eps'+2)<\eps$.
By Definition~\ref{D2}(ii), there exist $\rho\in(0,\eps')$,
$a\in A \cap{\B}_{\eps'}(\bar{x})$, $b\in B \cap{\B}_{\eps'}(\bar{x})$ and $u,v\in{X}$ such that conditions \eqref{D2-2} are satisfied with $\eps'$ in place of $\eps$.
Then $\max\{\|a-\bx\|,\|b-\bx\|\}+\rho(\eps'+1) <\eps'+\eps'(\eps'+1)<\eps$, and it follows from Theorem~\ref{KL}(i) that there exist points
$a'\in A\cap\B_\eps(\bar{x})$, $b'\in B\cap\B_\eps(\bar{x})$, $a^*\in{N}_{A}(a')$ and $b^*\in{N}_B(b')$ satisfying \eqref{EP-2}, i.e. condition (ii) in Theorem~\ref{EP} holds true.

Conversely, let condition (ii) in Theorem~\ref{EP} holds true and an $\eps>0$ be given.
Then there exist points
$a\in A\cap\B_\eps(\bar{x})$, $b\in B\cap\B_\eps(\bar{x})$, $a^*\in{N}_{A}(a)$ and $b^*\in{N}_B(b)$ satisfying conditions \eqref{EP-2}.
By Theorem~\ref{KL}(ii), there exists a $\rho\in(0,\eps)$ and points $u,v\in{X}$ satisfying conditions \eqref{D2-2}, i.e. the pair $\{A,B\}$ is approximately stationary at $\bx$.
\qed\end{proof}

\begin{remark}
The observations in Remark~\ref{R3.1} are applicable to Theorem~\ref{KL}(i): we can always assume that $\eps<1$.
Similarly, in Theorem~\ref{ZN} we can assume that $\eps<\la$.
\end{remark}

Next we compare the statements of Theorem~\ref{KL}(i) and Theorem~\ref{ZN}.
There are important similarities between them: both establish a kind of generalized separation of the two sets, related somehow to the given pair of points $a\in A$ and $b\in B$ possessing a certain approximate `extremality' property.
There are also essential differences.

We start with comparing the assumptions in the two statements.
On the first glance, they look mutually exclusive: the first one assumes the existence of a point $\bx\in A\cap B$, while in the second theorem, it is assumed on the contrary that $A\cap B=\emptyset$.
However, this distinction is easy to overcome.
Given points $a\in A$ and $b\in B$ in Theorem~\ref{ZN}, one can set $A':=A-a$ and $B':=B-b$;
then $\bx:=0\in A'\cap B'$ (this trick is used in the proof of Theorem~\ref{ZN}$^\prime$ below).
This observation exposes also the different roles played by the pairs $a\in A$ and $b\in B$ in Theorem~\ref{KL}(i) and Theorem~\ref{ZN}.
In the first one, these are actually additional parameters having no analogues in Theorem~\ref{ZN}, which corresponds to $a=b=\bx$ in Theorem~\ref{KL}(i).

The second distinction is related to the main approximate `extremality' assumptions on the pair of sets: conditions \eqref{D2-2} in Theorem~\ref{KL}(i) and condition \eqref{R2.4-1} in Theorem~\ref{ZN}.
The next proposition shows that condition \eqref{R2.4-1} implies a stronger version of conditions \eqref{D2-2}.

\begin{proposition}[Conditions \eqref{D2-2} vs condition \eqref{R2.4-1}]\label{P1}
Suppose $X$ is a normed linear space, $A,B\subset X$, $A\cap B=\emptyset$ and $\eps>0$.
If points $a\in A$ and $b\in B$ satisfy condition \eqref{R2.4-1}, then there exist $u,v\in X$ such that $\|u\|=\|v\|<\frac{\eps}{2}$ and
\begin{equation}\label{P1-1}
(A-a-u)\cap(B-b-v)=\emptyset.
\end{equation}
Moreover, one can take
\begin{equation}\label{P1-2}
u:=\frac{\eps'}{2}\cdot\frac{b-a}{\norm{b-a}} \qdtx{and} v:=\frac{\eps'}{2}\cdot\frac{a-b}{\norm{b-a}},
\end{equation}
where, if $d(A,B)>0$, then $\eps'$ can be any number satisfying $\norm{b-a}-d(A,B)<\eps'<\min\{\eps,\norm{b-a}\}$, and if $d(A,B)=0$, then $\eps'=\norm{b-a}<\eps$.
\end{proposition}

\begin{proof}
Suppose $u$ and $v$ are given by \eqref{P1-2} and condition \eqref{P1-1} is violated.
Then
$\hat a-a-u=\hat b-b-v$ for some $\hat a\in A$ and $\hat b\in B$, and
\begin{gather}\label{P1-3}
\hat b-\hat a=b-a-(v-u) =\left(1-\frac{\eps'}{\norm{b-a}}\right)(b-a).
\end{gather}
If $d(A,B)>0$, then $\eps'<\norm{b-a}$, and consequently,
\begin{gather*}
\norm{b-a}-d(A,B)\ge\norm{b-a}-\norm{\hat b-\hat a} =\norm{b-a}-\left(1-\frac{\eps'}{\norm{b-a}}\right)\norm{b-a}=\eps',
\end{gather*}
which contradicts the choice of $\eps'$.
If $d(A,B)=0$, then $\eps'=\norm{b-a}$, and it follows from \eqref{P1-3} that $\hat b=\hat a$ which contradicts the assumption that $A\cap B=\emptyset$.
Thus, condition \eqref{P1-1} is true.
\qed\end{proof}

Proposition~\ref{P1} is not reversible: condition \eqref{P1-1} being satisfied with some small $u$ and $v$ does not imply that $\|a-b\|$ is close to the distance $d(A,B)$ between the two sets.

\begin{example}\label{E3.2}
Let $A:=\{(x_1,x_2)\in\R^2\mid x_2\le0\}$ and $B:=\{(x_1,x_2)\in\R^2\mid x_2\ge1\}$.
Then, assuming that $\R^2$ is equipped with e.g. the sum norm, $d(A,B)=1$.
If $a:=(\al,0)\in A$ and $b:=(\be,1)\in B$ with some $\al,\be\in\R$, then condition \eqref{P1-1} is satisfied with $u:=(0,\eps)$, $v:=(0,-\eps)$ and any $\eps>0$.
At the same time, $\|a-b\|=|\al-\be|+1$ can be arbitrarily large when the numbers $\al$ and $\be$ are far apart.
\end{example}

Thus, condition \eqref{P1-1}
with small $u$ and $v$ is less restrictive than condition \eqref{R2.4-1}.
Moreover, the first condition in \eqref{D2-2} with $\rho<\infty$ is weaker than \eqref{P1-1} and allows for local versions of the corresponding properties.

The next assertion is immediate from Proposition~\ref{P1}.

\begin{corollary}\label{C3.3}
Suppose $X$ is a normed linear space, $A,B\subset X$, $A\cap B=\emptyset$.
If sequences $\{a_k\}\subset A$ and $\{b_k\}\subset B$ are such that $\|a_k-b_k\|\to d(A,B)$, then there exist sequences $\{u_k\},\{v_k\}\subset X$ converging to 0, such that
\begin{equation*}
(A-a_k-u_k)\cap(B-b_k-v_k)=\emptyset.
\end{equation*}
\end{corollary}

With the sets in Example~\ref{E3.2}, one can easily see that the statement of Corollary~\ref{C3.3} is not reversible.

Now we are going to compare the conclusions of the two theorems.
Similarly to the two conditions in Theorem~\ref{EP}, they represent two different ways of formulating dual extremality/separation conditions:
in terms of normal (in Theorem~\ref{KL}(i)) or `almost normal' (in Theorem~\ref{ZN}) vectors, with the connection between the two formulations provided by Lemma~\ref{L1}.
However, unlike the two equivalent conditions in Theorem~\ref{EP} formulated `for any $\eps>0$', in both Theorem~\ref{KL}(i) and Theorem~\ref{ZN} the number $\eps>0$ is a given quantitative parameter.
Lemma~\ref{L1} used in the proof of the equivalence of the two conditions in Theorem~\ref{EP} cannot provide one-to-one translation between the two settings with the given $\eps>0$; it only gives estimates, and its application leads to some `loss of accuracy'.
Note that the proof of \cite[Theorem~3.1]{KruLop12.1}, where Theorem~\ref{KL} is taken from, contains estimates in terms of `almost normal' vectors and then employs the arguments used in the proof of Lemma~\ref{L1}(ii) to ensure that the vectors belong to the normal cones.
To make a fair comparison, one needs to either reformulate Theorem~\ref{ZN} in terms of normal vectors using Lemma~\ref{L1}, or extract the pre-Lemma~\ref{L1} statement from the proof of \cite[Theorem~3.1]{KruLop12.1}.
Below for simplicity we follow the first approach.
The next statement is a consequence of Theorem~\ref{ZN} and Lemma~\ref{L1}(ii).
\medskip

\noindent\bf{Theorem~\ref{ZN}$^\prime$}
\it Suppose $X$ is an Asplund space, $A,B\subset X$ are closed, $A\cap B=\emptyset$.
If points $a\in A$ and $b\in B$ satisfy condition \eqref{R2.4-1} with some $\eps>0$,
then, for any $\la>0$, there exist points
$a'\in A\cap\B_{\frac{\eps+\la}{2}}(a)$, $b'\in B\cap\B_{\frac{\eps+\la}{2}}(b)$, $a^*\in N_A(a')$ and $b^*\in N_B(b')$ such that
\rm
\begin{gather}\label{ZN-2'}
\|a^*\|+\|b^*\|=1
\qdtx{and}
\|a^*+b^*\|<\eps/\la.
\end{gather}
\medskip

\begin{proof}
Given $\eps>0$ and $\la>0$, set $\la':=\frac{\eps+\la}{2}$.
Thanks to Theorem~\ref{ZN}, there exist points
$a'\in A\cap\B_{\la'}(a)$, $b'\in B\cap\B_{\la'}(b)$ and $a^*\in{X}^*$ satisfying conditions \eqref{ZN-2} with $\la'$ in place of $\la$.
Then
\begin{gather*}
\norm{\frac{a^*}{2}}+\norm{-\frac{a^*}{2}}=1\qdtx{and}
d\left(\frac{a^*}{2},N_A(a')\right) +d\left(-\frac{a^*}{2},N_B(b')\right) <\frac{\varepsilon}{2\lambda}.
\end{gather*}
Using Lemma~\ref{L1}(ii), we can find $\hat a^*\in N_A(a')$ and $\hat b^*\in N_B(b')$ such that $$\|\hat a^*\|+\|\hat b^*\|=1\qdtx{and}
\|\hat a^*+\hat b^*\|<\frac{\eps/(2\la')}{1-\eps/(2\la')} =\frac{\eps}{2\la'-\eps}=\frac{\eps}{\la}.$$
\qed\end{proof}

Now the comparison is straightforward.

\begin{proposition}\label{P3.4}
Theorem~{\rm \ref{KL}(i)} implies Theorem~{\rm \ref{ZN}}$^\prime$.
\end{proposition}

\begin{proof}
Under the conditions of Theorem~{\rm \ref{ZN}}$^\prime$,
set $A':=A-a$ and $B':=B-b$.
Then $0\in A'\cap B'$.
By Proposition~\ref{P1}, there exist $u,v\in X$ such that $\|u\|=\|v\|<\frac{\eps}{2}$ and
$(A'-u)\cap(B'-v)=\emptyset$.
Choose an $\eps'\in(2\|u\|,\eps)$ and set $\rho:=\la/2$, $\hat\eps:=\eps'/\la$ and $\de:=(\eps+\la)/2$.
Then $\hat\eps<\eps/\la$ and $\de>\rho(\hat\eps+1)$.
By Theorem~\ref{KL}(i) applied to the sets $A'$ and $B'$ at 0, there exist points
$a'\in A\cap\B_\de(a)$, $b'\in B\cap\B_\de(b)$, $a^*\in{N}_{A}(a')$ and $b^*\in{N}_B(b')$ satisfying conditions \eqref{ZN-2'}.
\qed\end{proof}

Thus, Theorem~\ref{ZN}$^\prime$ is a special case of Theorem~\ref{KL}(i).
On the other hand, as demonstrated in \cite{ZheNg06,ZheNg11}, Theorem~\ref{ZN} (as well as its version formulated above as Theorem~\ref{ZN}$^\prime$) is sufficient for many important applications.
Next we show that Theorem~\ref{ZN}$^\prime$ implies the nonlocal version of the extremal principle.

\begin{corollary}[Nonlocal extremal principle]\label{NLEP}
Suppose $X$ is an Asplund space, $A,B\subset X$ are closed and $\bx\in A\cap B$.
If the pair $\{A,B\}$ is extremal at $\bx$, then the two equivalent conditions in Theorem~\ref{EP} hold true.
\end{corollary}

\begin{proof}
Let the pair $\{A,B\}$ be extremal and a number $\eps>0$ be given.
Choose an $\eps'\in\left(0,\frac{\eps^2}{\eps+1}\right)$.
Then $\frac{\eps'}{\eps}<\eps-\eps'$ and we can choose a $\la$ such that $\frac{2\eps'}{\eps}<\la<2(\eps-\eps')$.
There exist vectors $u,v\in{X}$
satisfying conditions \eqref{D1-1} with $\eps'$ in place of $\eps$.
Define $A':=A-u$, $B':=B-v$, $a:=\bx-u$ and $b:=\bx-v$.
Then $A'\cap B'=\es$ and $\norm{a-b}<2\eps'$.
Applying Theorem~\ref{ZN}$^\prime$, we find points $a'\in A$, $b'\in B$, $a^*\in N_{A'}(a'-u)=N_A(a')$ and $b^*\in N_{B'}(b'-v)=N_B(b')$ such that $\max\{\|a'-\bx\|,\|b'-\bx\|\}<\eps'+\la/2<\eps$, $\|a^*\|+\|b^*\|=1$
and
$\|a^*+b^*\|<2\eps'/\la<\eps$.
Thus, condition (ii) in Theorem~\ref{EP} is satisfied.
\qed\end{proof}



\begin{remark}\label{R4.2}
1. It is not difficult to modify the proof of Corollary~\ref{NLEP} to cater for the relaxed version of nonlocal extremality without the assumption $A\cap B\ne\es$ (see Remark~\ref{R2.4}).

2. Theorem~\ref{ZN} does not seem to be able to recapture the full local extremal principle (as in Theorem~\ref{EP}), not to say the extended extremal principle (as in Theorem~\ref{EEP}).

3. Condition \eqref{ZN-3} in Theorem~\ref{ZN} determining the `direction' of the vector $a^*$ does not have a direct analogue in the statement of Theorem~\ref{KL}(i).
Together with the first condition in \eqref{ZN-2}, it comes from subdifferentiating a norm at a nonzero point in the proof of \cite[Theorem~3.4]{ZheNg11}.
Subdifferentiating a norm is an essential component also in the proofs of the conventional extremal principle and all its modifications, including the one in \cite[Theorem~3.1]{KruLop12.1}; so analogues of \eqref{ZN-3} are implicitly present in all such proofs.
Zheng and Ng \cite{ZheNg11} seem to be the first to notice the importance of conditions like \eqref{ZN-3} for recapturing the classical convex separation theorem, and make \eqref{ZN-3} explicit in the statement of \cite[Theorem~3.4]{ZheNg11}.
In the current paper, keeping in line with the conventional formulations and for the sake of simplicity of the presentation, we will not formulate analogues of the condition \eqref{ZN-3} in the subsequent statements.
\end{remark}

\subsection{Relative extremality and stationarity}

The conventional definition of extremality (Definition~\ref{D1}) and most of its extensions presume that the sets have a common point.
However, as it was demonstrated in Subsection~\ref{FE}, there are natural situations which allow for and, in fact, require application of the extremal principle or its extensions to sets with empty intersection.
Such situations are formalized in the current subsection.

As it was observed in Remark~\ref{R2.4}, the nonlocal extremality property in Definition~\ref{D1}(i) does not use the assumption $A\cap B\ne\es$, and the conventional proof of the extremal principle can proceed without this assumption.
Now we are going to relax the definitions of local extremality and stationarity properties of pairs of sets.
Instead of considering both sets near a common point, we are going to consider each set near its own point.
The next definition builds on the simple trick employed in the proof of Proposition~\ref{P3.4} and present implicitly already in Definition~\ref{D2}(ii) of approximate stationarity and, in view of Proposition~\ref{P4.6} below, even in Definition~\ref{D1}.

\begin{definition}[Relative extremality and stationarity] \label{D4.1}
Suppose $X$ is a normed linear space, $A,B\subset X$, $a\in A$ and $b\in B$.
\begin{enumerate}
\item
The pair $\{A,B\}$ is
extremal relative to $a$ and $b$ if the pair $\{A-a,B-b\}$ is extremal, i.e.
for any $\varepsilon>0$ there
are $u,v\in X$ such that
\begin{gather*}
(A-a-u)\cap (B-b-v)=\emptyset
\qdtx{and}
\max\{\|u\|,\|v\|\}<\eps.
\end{gather*}
\item
The pair $\{A,B\}$ is
locally extremal relative to $a$ and $b$ if the pair $\{A-a,B-b\}$ is locally extremal at 0, i.e.
there exists a $\rho>0$ such that for any $\varepsilon>0$ there
are $u,v\in X$ such that
\begin{gather}\label{D4.1-2}
(A-a-u)\cap (B-b-v)\cap(\rho{\B})=\emptyset
\qdtx{and}
\max\{\|u\|,\|v\|\}<\eps.
\end{gather}
\item
The pair $\{A,B\}$ is
stationary relative to $a$ and $b$ if the pair $\{A-a,B-b\}$ is stationary at 0, i.e.
for any $\varepsilon>0$ there exist a $\rho\in(0,\eps)$ and
$u,v\in{X}$ such that conditions \eqref{D2-2} hold true.
\item
The pair $\{A,B\}$ is
approximately stationary relative to $a$ and $b$ if the pair $\{A-a,B-b\}$ is approximately stationary at 0, i.e., for any $\varepsilon>0$ there exist a $\rho\in(0,\eps)$ and points
$a'\in A \cap{\B}_\eps(a)$, $b'\in B \cap{\B}_\eps(b)$ and $u,v\in{X}$ such that conditions \eqref{D2-2} with $a'$ and $b'$ in place of $a$ and $b$ hold true.
\end{enumerate}
\end{definition}

Definition~\ref{D4.1} reduces the extremality, local extremality, stationarity and approximate stationarity at individual points to the corresponding conventional properties in the sense of Definitions~\ref{D1} and \ref{D2}.
On the other hand,
Definitions~\ref{D1} and \ref{D2} are special cases of the corresponding items in Definition~\ref{D4.1} when $a=b=\bx$.
From Proposition~\ref{P2.2} we get the following statement.

\begin{proposition}[Extremality vs stationarity]\label{P4.1}
Suppose $X$ is a normed linear space, $A,B\subset X$, $a\in A$ and $b\in B$.
Consider the following properties:
\begin{enumerate}
\item
$\{A,B\}$ is extremal relative to $a$ and $b$;
\item
$\{A,B\}$ is locally extremal relative to $a$ and $b$;
\item
$\{A,B\}$ is stationary relative to $a$ and $b$;
\item
$\{A,B\}$ is approximately stationary relative to $a$ and $b$.
\end{enumerate}
Then
\rm(i) \folgt (ii) \folgt (iii) \folgt (iv).
If, additionally, $A$ and $B$ are convex, then
\rm(i) \iff (ii) \iff (iii) \iff (iv).
\end{proposition}

Similarly, the next two theorems generalizing the conventional extended extremal principle in Theorem~\ref{EEP} and its extension in Theorem~\ref{KL} to the case of individual points are direct corollaries of Theorems~\ref{EEP} and \ref{KL}, respectively.

\begin{theorem}[Relative extended extremal principle]\label{T4.3}
Suppose $X$ is an Asplund space, $A,B\subset X$ are closed, $a\in A$ and $b\in B$.
The pair $\{A,B\}$ is  approximately stationary relative to $a$ and $b$ if and only if the following two equivalent conditions hold:
\begin{enumerate}
\item
for any $\eps>0$ there exist points
$a'\in A\cap\B_\eps(a)$, $b'\in B\cap\B_\eps(b)$ and $a^*\in X^*$ satisfying conditions \eqref{EP-1} with $a'$ and $b'$ in place of $a$ and $b$, respectively;
\item
for any $\eps>0$ there exist points
$a'\in A\cap\B_\eps(a)$, $b'\in B\cap\B_\eps(b)$, $a^*\in{N}_{A}(a')$ and $b^*\in{N}_B(b')$ satisfying conditions \eqref{EP-2}.
\end{enumerate}
\end{theorem}

\begin{theorem}[`Relative' version of Theorem~\ref{KL}]\label{T4.2}
Suppose $X$ is an Asplund space, $A,B\subset X$ are closed, $a\in A$ and $b\in B$.
\begin{enumerate}
\item
If points $\hat a\in A$, $\hat b\in B$ and $u,v\in X$ satisfy conditions \eqref{D2-2} with $\hat a$ and $\hat b$ in place of $a$ and $b$, respectively, and
some $\eps>0$ and $\rho>0$,
then, for any $\de>\max\{\|\hat a-a\|,\|\hat b-b\|\}+ \rho(\eps+1)$, there exist points
$a'\in A\cap\B_\de(a)$, $b'\in B\cap\B_\de(b)$, $a^*\in{N}_{A}(a')$ and $b^*\in{N}_B(b')$ satisfying conditions \eqref{EP-2}.
\item
Assertion {\rm(ii)} in Theorem~\ref{KL} holds true.
\end{enumerate}
\end{theorem}

Note that the extremality in part (i) of Definition~\ref{D4.1} can be considered as a special case of the local extremality in part (ii) of that definition with $\rho=\infty$.
On the other hand, as the next proposition shows, the local extremality can be considered as a special case of the extremality for a special pair of `localized' sets.

\begin{proposition}[Extremality vs local extremality]
Suppose $X$ is a normed linear space, $A,B\subset X$, $a\in A$ and $b\in B$.
If the pair $\{A,B\}$ is locally extremal relative to $a$ and $b$ with some $\rho>0$, then, for any $\rho'\in(0,\rho)$, the pair $\{A\cap\B_{\rho'}(a),B\cap\B_{\rho'}(b)\}$ is extremal relative to $a$ and $b$.
\end{proposition}

\begin{proof}
Let the pair $\{A,B\}$ be locally extremal relative to $a$ and $b$ with some $\rho>0$, and numbers
$\eps>0$ and $\rho'\in(0,\rho)$ be given.
Choose an $\eps'\in(0,\min\{\eps,\rho-\rho'\})$.
Then there exist $u,v\in X$ satisfying conditions \eqref{D4.1-2} with $\eps'$ in place of $\eps$.
Hence, $\max\{\|u\|,\|v\|\}<\eps'<\eps$, $\rho'\B-u\subset\rho\B$, $\rho'\B-v\subset\rho\B$, and
\begin{align*}
(A\cap\B_{\rho'}(a)-a-u)&\cap(B\cap\B_{\rho'}(b)-b-v)
\\
&=(A-a-u)\cap(\rho'\B-u)\cap(B-b-v)\cap(\rho'\B-v)
\\
&\subset(A-a-u)\cap(B-b-v)\cap(\rho\B)=\es.
\end{align*}
Thus, the pair $\{A\cap\B_{\rho'}(a),B\cap\B_{\rho'}(b)\}$ is extremal relative to $a$ and $b$.
\qed\end{proof}

In view of Proposition~\ref{P4.1}, the next proposition shows that all the properties in Definition~\ref{D4.1} are meaningful only when $a\in\bd A$ and $b\in\bd B$.

\begin{proposition}[Approximate stationarity relative to boundary points]
Suppose $X$ is a normed linear space, $A,B\subset X$, $a\in A$ and $b\in B$.
If the pair $\{A,B\}$ is approximately stationary relative to $a$ and $b$, then $a\in\bd A$ and $b\in\bd B$.
\end{proposition}

\begin{proof}
Suppose, on the contrary, that $a\in\Int A$.
(The case $b\in\Int B$ is not much different.)
Then $\B_r(a)\subset A$ for some $r>0$.
Choose an $\eps\in(0,\min\{r/3,1\})$.
If $\rho\in(0,\eps)$,
$a'\in A \cap{\B}_\eps(a)$, $b'\in B \cap{\B}_\eps(b)$ and $u,v\in(\eps\rho)\B$, then $\hat x:=-v\in(B-b'-v)\cap(\rho\B)$ and $\|\hat x+a'+u-a\|\le \|a'-a\|+\|u\|+\|v\|<\eps(1+2\rho)<3\eps<r$.
Hence, $\hat x+a'+u\in A$ and $\hat x\in A-a'-u$.
It follows that $(A-a'-u)\cap(B-b'-v)\cap(\rho\B)\ne\es$, and consequently, the pair $\{A,B\}$ is not approximately stationary relative to $a$ and $b$.
\qed\end{proof}

The conventional Definitions~\ref{D1} and \ref{D2} of the extremality and stationarity properties of $\{A,B\}$ involve the translations $A-u$ and $B-v$ of the sets and, thus, refer implicitly to the `relative' versions of the corresponding properties.
The next proposition is in a sense a reformulation of Definition~\ref{D4.1}.

\begin{proposition}[Conventional vs relative extremality and stationarity]\label{P4.6}
Suppose $X$ is a normed linear space, $A,B\subset X$ and $\bx\in A\cap B$.
The pair $\{A,B\}$ is extremal/locally extremal/stationary/approximately stationary at $\bx$ if and only if, for any $u,v\in X$, the pair $\{A-u,B-v\}$ is extremal/locally extremal/stationa\-ry/approximately stationary relative to $\bx-u$ and $\bx-v$.
\end{proposition}

In the case of (local) extremality, the distance between the translated sets $A-u$ and $B-v$ does not have to be attained at the translated points $\bx-u$ and $\bx-v$; see Fig.~\ref{F3}.

\begin{figure}[!ht]
\centering
\includegraphics[height=2.5cm]{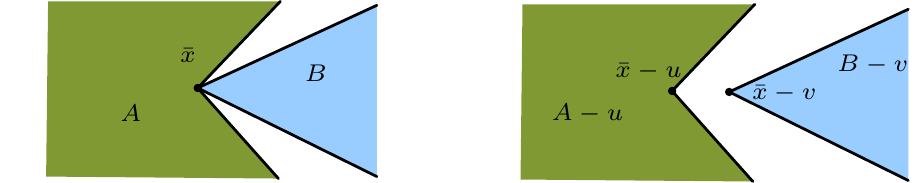}
\caption{Conventional vs relative extremality}\label{F3}
\end{figure}

The conventional Definitions~\ref{D1} and \ref{D2} of the extremality and stationarity properties as well as their extensions in Definition~\ref{D4.1} involve vectors $u,v$ determining the ``shifts'' of each of the sets.
It was observed in \cite{Kru05} that in the case of conventional extremality and stationarity of two sets it is sufficient to shift one of the sets only.
(In the general case of $n$ ($n>1$) sets, one can consider shifts of $n-1$ sets.)

\begin{proposition}[Relative extremality and stationarity with a single set shifted]\label{P4.5}
Suppose $X$ is a normed linear space, $A,B\subset X$, $a\in A$ and $b\in B$.
The pair $\{A,B\}$ is extremal/locally extremal/stationary/ap\-proximately stationary relative to $a$ and $b$ if and only if the respective conditions in Definition~\ref{D4.1} are satisfied with $v=0$.
\end{proposition}

\begin{proof}
If any of the conditions in Definition~\ref{D4.1} is satisfied with $v=0$, then the respective property obviously holds.

Conversely, if $(A-a-u)\cap(B-b-v)\cap(\rho\B)=\emptyset$ for some $\rho\in(0,\infty]$ and some $u,v\in X$, then $(A-a-u')\cap(B-b)\cap(\rho\B+v)=\emptyset$ where $u':=u-v$.
Set $\rho':=\rho/2$.
If $\max\{\|u\|,\|v\|\}<\al\le\rho/2$,
then $\|u'\|<2\al$ and $\rho'\B\subset\rho\B+v$.
Hence,
\begin{gather}\label{P4.5-1}
(A-a-u')\cap(B-b)\cap(\rho'\B)=\emptyset.
\end{gather}
These simple observations allow one to ensure each of the conditions in Definition~\ref{D4.1} with this $u'$ and appropriate choice of $\eps$ and $\rho$.

Let condition (ii) in Definition~\ref{D4.1} be satisfied with some $\rho\in(0,\infty]$.
(As observed above, the case $\rho=\infty$ covers condition (i) in Definition~\ref{D4.1}.)
Then, with $\rho':=\rho/2$ and any $\eps'>0$, one can take $\al:=\eps:=\min\{\eps',\rho\}/2$ and find $u,v\in X$ such that conditions \eqref{D4.1-2} hold.
With $u'$ defined as above, one has $\|u'\|<2\al\le\eps'$ and condition \eqref{P4.5-1} is satisfied; hence, condition (ii) (condition (i) if $\rho=\infty$) in Definition~\ref{D4.1} is satisfied with $v=0$.

Let condition (iii) (condition (iv)) in Definition~\ref{D4.1} be satisfied.
Then, with any $\eps'>0$, one can take $\eps:=\min\{\eps'/4,1/2\}$ and find a $\rho\in(0,\eps)$ and points $u,v\in X$ (and $a'\in A \cap{\B}_\eps(a)$ and $b'\in B \cap{\B}_\eps(b)$) such that conditions \eqref{D2-2} hold.
With $\rho':=\rho/2$, $\al:=\eps\rho\le\rho/2$ and $u'$ defined as above, one has $\rho'\in(0,\eps')$, $\|u'\|<2\al\le\eps'\rho'$ (and $a'\in A \cap{\B}_{\eps'}(a)$ and $b'\in B \cap{\B}_{\eps'}(b)$) and condition \eqref{P4.5-1} is satisfied (with $a'$ and $b'$ in place of $a$ and $b$); hence, condition (iii) (condition (iv)) in Definition~\ref{D4.1} is satisfied with $v=0$.
\qed\end{proof}

\begin{remark}
1. Condition $v=0$ in Proposition~\ref{P4.5} can be replaced with $u=0$.

2. Similarly to Proposition~\ref{P4.5}, one can also impose condition $v=0$ (or $u=0$) in Theorem~\ref{T4.2}.
However, unlike Proposition~\ref{P4.5}, this would require appropriate amendments in the estimates.

3. Thanks to Proposition~\ref{P4.5}, the relative extremality in Definition~\ref{D4.1}(i) is equivalent to the condition $0\in\bd[(A-a)-(B-b)]$.
\end{remark}

It was observed in \cite[Theorem~1]{Kru05} (see also \cite[Theorem~1]{Kru06}) that approximate stationarity in Definition~\ref{D2}(ii) can be characterized in metric terms.
The next proposition provides a version of this result for the relative approximate stationarity  in Definition~\ref{D4.1}(iv).
It is a consequence of \cite[Theorem~1(ii)]{Kru06} and Definition~\ref{D4.1}(iv).

\begin{proposition}[Metric characterization of approximate stationarity]
Suppose $X$ is a normed linear space, $A,B\subset X$, $a\in A$ and $b\in B$.
The pair $\{A,B\}$ is ap\-proximately stationary relative to $a$ and $b$ if and only if, for any $\eps>0$, there exist $y\in\B_\eps(a)$, $z\in\B_\eps(b)$ and $x\in\eps\B$ such that
\begin{gather*}
\max\{d(x,A-y),d(x,B-z)\}<\eps d(x,(A-y)\cap(B-z)).
\end{gather*}
\end{proposition}

The next proposition shows that, if the distance between $A$ and $B$
is attained (at least locally) at some points $a$ and $b$, then
the pair $\{A,B\}$ is (locally) extremal relative to $a$ and $b$.

\begin{proposition}[Relative extremality when the distance is attained]
\label{P4.2}
Suppose $X$ is a normed linear space, $A,B\subset X$ are closed, $a\in A$ and $b\in B$.
\begin{enumerate}
\item
If $\|a-b\|=d(A,B)>0$,
then
the pair $\{A,B\}$ is extremal relative to $a$ and $b$.
\item
If  $\|a-b\|=d(A\cap\B_\rho(a),B\cap\B_\rho(b))>0$
for some $\rho>0$, then
the pair $\{A,B\}$ is locally extremal relative to $a$ and $b$.
\end{enumerate}
\end{proposition}

\begin{proof}
(i) Take an arbitrary $\eps>0$ and a $t\in(0,\min\{\eps/\|a-b\|,1/2\})$.
Set $u:=t(b-a)$ and $v:=t(a-b)$.
Then $\max\{\|u\|,\|v\|\}=t\|a-b\|<\eps$, and we only need to show that
\begin{equation}\label{P4.2-2}
(A-a-u)\cap (B-b-v)=\emptyset.
\end{equation}
Suppose this is not true, i.e, there exists an $x\in X$ such that $a':=a+u+x\in A$ and $b':=b+v+x\in B$.
Then $\|a'-b'\|=(1-2t)\|a-b\|<\|a-b\|$, which contradicts the assumption.
Hence, condition \eqref{P4.2-2} is true and, consequently, the pair $\{A,B\}$ is extremal relative to $a$ and $b$.

(ii) Take a $\rho'\in(0,\rho)$, an arbitrary $\eps>0$ and a $t\in(0,\min\{\eps/\|a-b\|,(\rho-\rho')/\|a-b\|,1/2\})$.
Set $u:=t(b-a)$ and $v:=t(a-b)$.
Then $\max\{\|u\|,\|v\|\}=t\|a-b\|<\eps$, and we only need to show that
\begin{equation}\label{P4.2-1}
(A-a-u)\cap (B-b-v) \cap(\rho'{\B})=\emptyset.
\end{equation}
Suppose this is not true, i.e, there exists an $x\in\rho'{\B}$ such that $a':=a+u+x\in A$ and $b':=b+v+x\in B$.
Observe that $\max\{\|a'-a\|,\|b'-b\|\}\le t\|a-b\|+\rho'<\rho$.
Thus, $a'\in A\cap\B_\rho(a)$ and $b'\in B\cap\B_\rho(b)$, and $\|a'-b'\|=(1-2t)\|a-b\|<\|a-b\|$, which contradicts the assumption.
Hence, condition \eqref{P4.2-2} is true and, consequently, the pair $\{A,B\}$ is locally extremal relative to $a$ and $b$.
\qed\end{proof}

\begin{remark}
As the next example shows, the condition $d(A,B)>0$ in part (i) and the similar condition in part (ii) of Proposition~\ref{P4.2} cannot be dropped.
\end{remark}

\begin{example}
For the sets $A:=\{(x_1,x_2)\mid x_2\le0\}$ and $B:=\{(x_1,x_2)\mid x_1+x_2\le0\}$ in $\R^2$ we have $d(A,B)=0$, while $\{A,B\}$ is obviously not extremal (in fact it is not even approximately stationary) at any $a=b\in A\cap B$ (in particular, at $a=b=0$).
\end{example}

\begin{remark}
When $A\cap B=\es$ and the distance is not attained, there may or may not be a pair of points $a\in A$ and $b\in B$ such that the pair $\{A,B\}$ is (locally) extremal relative to $a$ and $b$.
\end{remark}

\begin{example}\label{E4.2}
1. For the sets $A:=\{(x_1,x_2)\mid x_2\le0\}$ and $B:=\{(x_1,x_2)\mid x_2\ge e^{-x_1}\}$ in $\R^2$ (see Fig.~\ref{F4}) we have $A\cap B=\es$, while $\{A,B\}$ is obviously not extremal (and even not approximately stationary) at any $a\in A$ and $b\in B$.
Note that $\{A,B\}$ could still be considered extremal if the conventional Definition~\ref{D1}(i) was amended as discussed in Remark~\ref{R2.4}.

2. For the sets in item 1 above, it holds $d(A,B)=0$.
This is not a precondition.
If the set $B$ above is translated upwards by one unit: $B:=\{(x_1,x_2)\mid x_2\ge e^{-x_1}+1\}$ (see Fig.~\ref{F5}), then $d(A,B)=1$ and $\{A,B\}$ is still not extremal at any $a\in A$ and $b\in B$.

3. If the two sets in item 1 above are modified slightly: $A:=\{(x_1,x_2)\mid -1\le x_2\le0\}$ and $B:=\{(x_1,x_2)\mid e^{-x_1}\le x_2\le1\}$ (see Fig.~\ref{F6}), then despite their intersection still being empty, the modified sets are extremal, e.g., at $a:=(1,-1)\in A$ and $b:=(1,1)\in B$.
\end{example}

\begin{figure}[!ht]
\centering
\begin{minipage}[b]{0.3\textwidth}
\centering
\includegraphics[height=2.5cm]{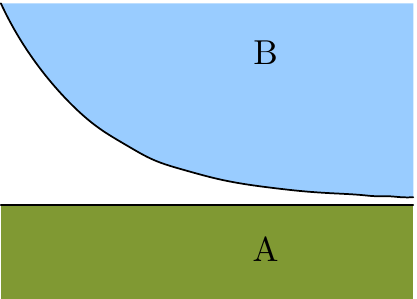}
\caption{Example~\ref{E4.2}.1}\label{F4}
\end{minipage}
\begin{minipage}[b]{0.3\textwidth}
\centering
\includegraphics[height=2.5cm]{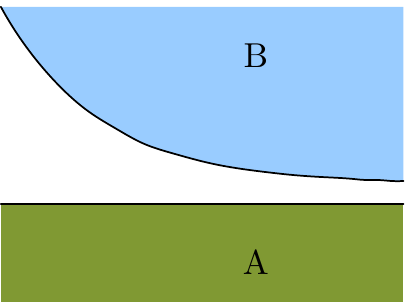}
\caption{Example~\ref{E4.2}.2}\label{F5}
\end{minipage}
\begin{minipage}[b]{0.3\textwidth}
\centering
\includegraphics[height=2.5cm]{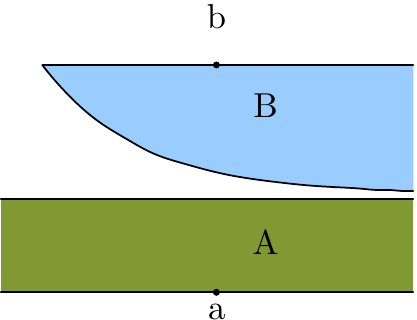}
\caption{Example~\ref{E4.2}.3}\label{F6}
\end{minipage}
\end{figure}

The approximate stationarity property in Definition~\ref{D4.1}(iii) possesses certain stability: if it holds at a certain pair $(a,b)$, it holds `approximately' at all nearby pairs.

\begin{proposition}[Stability of relative approximate stationarity]
\label{P4.3}
Suppose $X$ is a normed linear space, $A,B\subset X$ are closed, $a\in A$ and $b\in B$.
The pair $\{A,B\}$ is approximately stationary relative to $a$ and $b$ if and only if for any $\eps>0$, $\de>0$ and any points
$a'\in A\cap\B_\eps(a)$, $b'\in B\cap\B_\eps(b)$ there exist $\rho\in(0,\de)$, $a''\in A\cap\B_\eps(a')$, $b''\in B\cap\B_\eps(b')$ and $u,v\in{X}$ such that conditions \eqref{D2-2} hold true with $a''$ and $b''$ in place of $a$ and $b$.
\end{proposition}

\begin{proof}
The sufficiency is obvious: taking $a':=a$ and $b':=b$ in the conditions of Proposition~\ref{P4.3}, one satisfies the conditions in Definition~\ref{D4.1}(iii).
To prove the necessity,
let the pair $\{A,B\}$ be approximately stationary relative to $a$ and $b$, $\eps>0$, $\de>0$, $a'\in A\cap\B_\eps(a)$ and $b'\in B\cap\B_\eps(b)$.
Choose a $\xi\in(0,\de)$ such that $\max\{\|a'-a\|,\|b'-b\|\}+\xi<\eps$.
By Definition~\ref{D4.1}, there exist $\rho\in(0,\xi)$,
$a''\in A\cap{\B}_\xi(a)$, $b''\in B\cap{\B}_\xi(b)$ and $u,v\in{X}$ such that conditions \eqref{D2-2} hold true with $a''$ and $b''$ in place of $a$ and $b$.
Then $\rho<\de$, $\|a''-a'\|\le\|a''-a\|+\|a-a'\|<\eps$ and, similarly, $\|b''-b'\|<\eps$.
\qed\end{proof}

\subsection{Pairs of sets and \SVM s}\label{SS4.2}

It is well known (see e.g. \cite{Iof00_,KruTha15}) that regularity/transversality properties of collections of sets are in a sense equivalent to the corresponding regularity and Lipschitz-like properties of certain \SVM s.
Given two subsets $A$ and $B$ of a normed linear space $X$, the mappings $F:X\rightrightarrows X\times X$ and $S:X\times X\rightrightarrows X$ defined by
\begin{equation}\label{4.22}
F(x):=(A-x)\times(B-x),\;\;x\in X,
\qdtx{and}
S(y,z):=(A-y)\cap(B-z),\;\;y,z\in X
\end{equation}
play a major role in this type of analysis.
Below we establish links between the relative extremality and stationarity properties of the pair of sets $\{A,B\}$ and certain properties of the \SVM s $F$ and $S$ given by \eqref{4.22}.

First notice that $S=F\iv$, $\dom S=F(X)$, $F(0)=A\times B$ and $S(0,0)=A\cap B$.
Given an $a\in A$ and a $b\in B$, we obviously have $(a,b)\in F(0)$ and $0\in S(a,b)$.
Recall that in this paper we are assuming that the product space $X\times X$ is equipped with the maximum norm.
This corresponds to $\max\{\|u\|,\|v\|\}$ involved in all parts of Definition~\ref{D4.1}, and
is not a big restriction: any other norm compatible with the norm on $X$ can be used instead, as long as it is consistently used everywhere (cf. Remark~\ref{R2.3}).

\begin{proposition}[Pairs of sets and \SVM s]
\label{P4.7}
Suppose $X$ is a normed linear space, $A,B\subset X$, $a\in A$, $b\in B$ and \SVM s $F$ and $S$ are given by \eqref{4.22}.
\begin{enumerate}
\item
$\{A,B\}$ is extremal relative to $a$ and $b$ if and only if $(a,b)\in\bd F(X)$, or equivalently, $(a,b)\in\bd\dom S$.
\item
$\{A,B\}$ is locally extremal relative to $a$ and $b$ if and only if there exists a $\rho>0$ such that \begin{enumerate}
\item
$(a,b)\in\bd F(\rho\B)$, or equivalently,
\item
for any $\eps>0$, there is a pair $(y,z)\in\B_\eps(a,b)$ such that $S(y,z)\cap(\rho\B)=\es$.
\end{enumerate}
\item
$\{A,B\}$ is stationary relative to $a$ and $b$ if and only if, for any $\eps>0$, there exists a $\rho\in(0,\eps)$ such that
\begin{enumerate}
\item
$\B_{\eps\rho}(a,b)\not\subset F(\rho\B)$, or equivalently,
\item
there is a pair $(y,z)\in\B_{\eps\rho}(a,b)$ such that $S(y,z)\cap(\rho\B)=\es$.
\end{enumerate}
\item
$\{A,B\}$ is approximately stationary relative to $a$ and $b$ if and only if, for any $\eps>0$, there exist a $\rho\in(0,\eps)$ and points $a'\in A \cap{\B}_\eps(a)$, $b'\in B \cap{\B}_\eps(b)$ such that
\begin{enumerate}
\item
$\B_{\eps\rho}(a',b')\not\subset F(\rho\B)$, or equivalently,
\item
there is a pair $(y,z)\in\B_{\eps\rho}(a',b')$ such that $S(y,z)\cap(\rho\B)=\es$.
\end{enumerate}
\end{enumerate}
\end{proposition}

\begin{proof}
(i) By Definition~\ref{D4.1}(i), $\{A,B\}$ is extremal relative to $a$ and $b$ if and only if, for any $\eps>0$, there is a pair $(y,z)\in\B_{\eps}(a,b)$ such that $S(y,z)=(A-y)\cap(B-z)=\es$, i.e. $(a,b)\in\bd\dom S=\bd F(X)$.

(ii) By Definition~\ref{D4.1}(ii), $\{A,B\}$ is locally extremal relative to $a$ and $b$ if and only if there exists a $\rho>0$ such that for any $\varepsilon>0$ there is a pair $(y,z)\in\B_{\eps}(a,b)$ such that $S(y,z)\cap(\rho\B)=(A-y)\cap(B-z)\cap(\rho\B)=\es$.
This is equivalent to $(a,b)\in\bd F(\rho\B)$.

(iii) By Definition~\ref{D4.1}(iii), $\{A,B\}$ is stationary relative to $a$ and $b$ if and only if, for any $\eps>0$, there exist a $\rho\in(0,\eps)$ and a pair $(y,z)\in\B_{\eps\rho}(a,b)$ such that $S(y,z)\cap(\rho\B)=(A-y)\cap(B-z)\cap(\rho\B)=\es$.
This is equivalent to $\B_{\eps\rho}(a,b)\not\subset F(\rho\B)$.

(iv) By Definition~\ref{D4.1}(iv), $\{A,B\}$ is approximately stationary relative to $a$ and $b$ if and only if, for any $\eps>0$, there exist a $\rho\in(0,\eps)$, points $a'\in A \cap{\B}_\eps(a)$, $b'\in B \cap{\B}_\eps(b)$ and a pair $(y,z)\in\B_{\eps\rho}(a',b')$ such that $S(y,z)\cap(\rho\B)=(A-y)\cap(B-z)\cap(\rho\B)=\es$.
This is equivalent to $\B_{\eps\rho}(a',b')\not\subset F(\rho\B)$.
\qed\end{proof}

\begin{remark}
\label{R4.6}
The extremality and stationarity properties of pairs of sets studied in the current paper are in a sense examples of their `irregular behaviour'; cf. \cite{Kru05,Kru06,KruLop12.1}.
No surprise, the properties of the \SVM s $F$ and $S$ that appear in Proposition~\ref{P4.7} are in fact negations of certain regularity, semicontinuity and Lipschitz-like properties, some of which are well known.
Below we briefly comment on these properties.
\begin{enumerate}
\item
In accordance with Proposition~\ref{P4.7}(i),
\emph{$\{A,B\}$ is NOT extremal relative to $a$ and $b$ if and only if there exists an $\al>0$ such that $\B_\al(a,b)\subset F(X)$, or equivalently, $\B_\al(a,b)\subset\dom S$.}
This means that $F$ \emph{covers} $(a,b)$ (on $X$).
\item
In accordance with Proposition~\ref{P4.7}(ii),
\emph{$\{A,B\}$ is NOT locally extremal relative to $a$ and $b$ if and only if, for any $\rho>0$, there exists an $\al>0$ such that
\begin{enumerate}
\item
$\B_\al(a,b)\subset F(\rho\B)$, or equivalently, \item
$d(0,S(y,z))\le\rho$ for any $(y,z)\in\B_{\al}(a,b)$.
\end{enumerate}}
This means that $F$ \emph{covers} $(a,b)$ on $\rho\B$ (is \emph{open} at $(0,(a,b))$ \cite[p.~180]{DonRoc14}) and $S$ is \lsc\ at $((a,b),0)$ \cite[p.~10]{KlaKum02}.
\item
In accordance with Proposition~\ref{P4.7}(iii),
\emph{$\{A,B\}$ is NOT stationary relative to $a$ and $b$ if and only if there exists an $\al>0$ such that
\begin{enumerate}
\item
$\B_{\al\rho}(a,b)\subset F(\rho\B)$ for some $\de>0$ and all $\rho\in(0,\de)$, or equivalently,
\item
$d(0,S(y,z))\le\rho$ for some $\de>0$ and all $\rho\in(0,\de)$, $(y,z)\in\B_{\al\rho}(a,b)$.
\end{enumerate}}
Condition (a) means that $F$ \emph{$\al$-covers} \cite[p.~1765]{Kru09} (is \emph{$\al$-open} \cite[Definition~2.4(i)]{ApeDurStr13}) at $(0,(a,b))$.
It is equivalent \cite[Theorem~6(i)]{Kru09} to the following condition:
\begin{enumerate}
\item[(c)]
\emph{$\al d(0,F\iv(y,z))\le d((y,z),(a,b))$ for all $(y,z)$ near $(a,b)$,}
\end{enumerate}
which means that $F$ is \emph{semiregular} \cite[p.~1765]{Kru09}, \cite[Definition~1.2]{CibFabKru} (\emph{hemiregular} \cite[Definition~2.4(iii)]{ApeDurStr13}) at $(0,(a,b))$ with rank $\al$.

In its turn, condition (b) can be rewritten as
\begin{enumerate}
\item[(d)]
\emph{$\al d(0,S(y,z))\le d((y,z),(a,b))$ for all $(y,z)$ near $(a,b)$,}
\end{enumerate}
which means that $S$ is \emph{Lipschitz \lsc} \cite[p.~34]{KlaKum02} (\emph{pseudocalm} \cite[Definition~2.4(ii)]{ApeDurStr13}) at $(0,(a,b))$ with rank $\al$.
\item
In accordance with Proposition~\ref{P4.7}(iv),
\emph{$\{A,B\}$ is NOT approximately stationary relative to $a$ and $b$ if and only if there exists an $\al>0$ such that
\begin{enumerate}
\item
$\B_{\al\rho}(a',b')\subset F(\rho\B)$ for some $\de>0$ and all $\rho\in(0,\de)$, $a'\in A \cap{\B}_\de(a)$, $b'\in B \cap{\B}_\de(b)$, or equivalently,
\item
$d(0,S(y,z))\le\rho$ for some $\de>0$ and all $\rho\in(0,\de)$, $a'\in A \cap{\B}_\de(a)$, $b'\in B \cap{\B}_\de(b)$, $(y,z)\in\B_{\al\rho}(a',b')$.
\end{enumerate}}
It is not difficult to check that the above conditions are equivalent, respectively, to the following two (with possibly a smaller $\de$):
\emph{
\begin{enumerate}
\item[\rm(a$^\prime$)]
$\B_{\al\rho}(y,z)\subset F(\B_\rho(w))$ for some $\de>0$ and all $\rho\in(0,\de)$, $(w,y,z)\in\gph F\cap{\B}_\de(0,a,b)$, or equivalently,
\item[\rm(b$^\prime$)]
$d(w,S(y',z'))\le\rho$ for some $\de>0$ and all $\rho\in(0,\de)$, $(w,y,z)\in\gph F\cap{\B}_\de(0,a,b)$, $(y',z')\in\B_{\al\rho}(y,z)$.
\end{enumerate}}
Condition (a$^\prime$) means that $F$ \emph{uniformly covers} \cite[p.~1766]{Kru09} (is \emph{open with/at linear rate} \cite[p.~13]{KlaKum02}, \cite[Definition~2.1(i)]{ApeDurStr13}) around $(0,(a,b))$.
It is known to be equivalent (see e.g. \cite[Theorem~6(iii)]{Kru09}) to the following condition:
\begin{enumerate}
\item[(c)]
\emph{$\al d(w,F\iv(y,z))\le d((y,z),F(w))$ for all $(y,z)$ near $(a,b)$ and $w$ near 0,}
\end{enumerate}
which means that $F$ is \emph{metrically regular} \cite[p.~12]{KlaKum02}, \cite[p.~178]{DonRoc14} at $(0,(a,b))$ with rank $\al$.

In its turn, condition (b) can be rewritten as
\begin{enumerate}
\item[(d)]
\emph{$\al d(w,S(y',z'))\le \|(y',z')-(y,z)\|$ for all $(y,z)$ and $(y',z')$ near $(a,b)$ and $w\in S(y,z)$ near 0,}
\end{enumerate}
which means that $S$ has the \emph{Aubin property} \cite[p.~172]{DonRoc14} (is \emph{pseudo Lipschitz} \cite[(D1)]{KlaKum02}) at $(0,(a,b))$ with rank $\al$.
\end{enumerate}
There is little consistency in the literature about whether to put $\al$ in the left or the \RHS\ of the corresponding inequality/inclusion in the definitions of the properties discussed above.
Thus, in some sources it is $\al\iv$ that is taken as the quantitative estimate (rank, modulus) of the respective property instead of $\al$.
\end{remark}

Observe from \eqref{4.22} that $F(x)=A\times B - (x,x)$.
This simple observation provides a link between the extremality and stationarity properties of pairs of sets studied in the current paper and the \emph{nonconvex separation property} introduced by Borwein and Jofre \cite{BorJof98}.
Given subsets $A,B\subset X$ and points $a\in A$ and $b\in B$, define another pair of sets and a pair of points
\begin{equation}\label{4.23}
\tilde A:=A\times B,\quad \tilde B:=\{(x,x)\mid x\in X\} \qdtx{and}
\tilde a:=(a,b)\in\tilde A,\quad \tilde b:=(0,0)\in\tilde B
\end{equation}
in the product space $X\times X$.
Note that $\tilde B=-\tilde B$.
As previously, this space is assumed to be equipped with the maximum norm.
We are going to keep the standard notation $\B$ for the unit ball in $X\times X$.
The next proposition is a consequence of
Proposition~\ref{P4.7}.

\begin{proposition}[Pairs of sets in $X$ and $X\times X$]\label{P4.8}
Suppose $X$ is a normed linear space, $A,B\subset X$, $a\in A$, $b\in B$ and the sets $\tilde A$ and $\tilde B$ and points $\tilde a$ and $\tilde b$ are given by \eqref{4.23}.
\begin{enumerate}
\item
$\{A,B\}$ is extremal relative to $a$ and $b$ if and only if $\tilde a+\tilde b\in\bd(\tilde A+\tilde B)$.
\item
$\{A,B\}$ is locally extremal relative to $a$ and $b$ if and only if there exists a $\rho>0$ such that $\tilde a+\tilde b\in\bd(\tilde A+\tilde B \cap(\rho\B))$.
\item
$\{A,B\}$ is stationary relative to $a$ and $b$ if and only if, for any $\eps>0$, there exists a $\rho\in(0,\eps)$ such that
$\B_{\eps\rho}(\tilde a+\tilde b)\not\subset\tilde A+\tilde B \cap(\rho\B)$.
\item
$\{A,B\}$ is approximately stationary relative to $a$ and $b$ if and only if, for any $\eps>0$, there exist a $\rho\in(0,\eps)$ and a point $\tilde a'\in\tilde A \cap{\B}_\eps(\tilde a)$ such that
$\B_{\eps\rho}(\tilde a'+\tilde b)\not\subset\tilde A+\tilde B \cap(\rho\B)$.
\end{enumerate}
\end{proposition}

The condition $\tilde a+\tilde b\in\bd(\tilde A+\tilde B)$ in Proposition~\ref{4.23}(i) is the \emph{boundary condition} for the sets $\tilde A$ and $\tilde B$ introduced and characterized in \cite[Theorem~1]{BorJof98}, while the assertions in Proposition~\ref{4.23}(i) and (ii) in the special case $a=b$ improve \cite[Proposition~2(ii)]{BorJof98} (see \cite[Remark~9]{Kru05}).
The conditions in parts (iii) and (iv) of Proposition~\ref{4.23} define certain stationarity properties for the sets $\tilde A$ and $\tilde B$ which may be of independent interest.

\begin{remark}
Proposition~\ref{4.23} relates the extremality and stationarity properties of pairs of sets in $X$ with the corresponding `boundary condition'-like properties of certain pairs of sets in $X\times X$.
There is also a way in the opposite direction: given subsets $A,B\subset X$ and points $a\in A$ and $b\in B$, one can consider the boundary condition $a+b\in\bd(A+B)$ and its analogues as in the corresponding parts of Proposition~\ref{4.23} and relate them with the corresponding extremality and stationarity properties of the sets $\tilde A:=A\times B$ and $\tilde B:=\{(y,z)\mid y+z=a+b\}$ in $X\times X$ at the point $(a,b)$ along the lines of \cite[Proposition~2(i)]{BorJof98} and \cite[Section~4]{Kru05}).
\end{remark}

\begin{acknowledgements}
The authors thank the referees for careful reading of the manuscript and their constructive comments and suggestions.
\end{acknowledgements}

\addcontentsline{toc}{section}{References}
\def\cprime{$'$} \def\cftil#1{\ifmmode\setbox7\hbox{$\accent"5E#1$}\else
  \setbox7\hbox{\accent"5E#1}\penalty 10000\relax\fi\raise 1\ht7
  \hbox{\lower1.15ex\hbox to 1\wd7{\hss\accent"7E\hss}}\penalty 10000
  \hskip-1\wd7\penalty 10000\box7} \def\cprime{$'$} \def\cprime{$'$}
  \def\cprime{$'$} \def\cprime{$'$} \def\cprime{$'$}
  \def\Dbar{\leavevmode\lower.6ex\hbox to 0pt{\hskip-.23ex \accent"16\hss}D}
  \def\cfac#1{\ifmmode\setbox7\hbox{$\accent"5E#1$}\else
  \setbox7\hbox{\accent"5E#1}\penalty 10000\relax\fi\raise 1\ht7
  \hbox{\lower1.15ex\hbox to 1\wd7{\hss\accent"13\hss}}\penalty 10000
  \hskip-1\wd7\penalty 10000\box7} \def\cprime{$'$}

\end{document}

%% file: macro.tex
\newcommand{\al}{\alpha}
\newcommand{\be}{\beta}

\newcommand{\la}{\lambda}
\newcommand{\de}{\delta}
\newcommand{\eps}{\varepsilon}
\newcommand{\bx}{\bar x}

\newcommand{\iv}{^{-1} }

\newcommand {\R} {\mathbb R}
\newcommand {\N} {\mathbb N}

\newcommand {\B} {\mathbb B}

\newcommand {\gph} {{\rm gph}\,}
\newcommand {\dom} {{\rm dom}\,}

\newcommand {\bd} {{\rm bd}\,}

\newcommand {\Int} {{\rm int}\,}

\renewcommand{\iff}{$ \Leftrightarrow\ $}
\newcommand{\folgt}{$ \Rightarrow\ $}

\def\nbh{neighbourhood}
\def\es{\emptyset}
\def\lsc{lower semicontinuous}

\def\RHS{right-hand side}
\def\SVM{set-valued mapping}
\def\EVP{Ekeland variational principle}

\newcommand{\norm}[1]{\left\Vert#1\right\Vert}

\newcommand{\qdtx}[1]{\quad\mbox{#1}\quad}
\newcounter{mycount}